\documentclass[11pt, twoside]{article}
\usepackage{amsfonts}

\usepackage{amsmath,amssymb}
\usepackage{multicol}
\usepackage{multicol,graphics}
\usepackage{epsf}     
\usepackage{cite}

\newtheorem{theorem}{Theorem}[section]
\newtheorem{lemma}[theorem]{Lemma}
\newtheorem{corollary}[theorem]{Corollary}

\newtheorem{definition}[theorem]{Definition}
\newtheorem{remark}[theorem]{Remark}
\newtheorem{proposition}[theorem]{Proposition}
\numberwithin{equation}{section}
\newenvironment{proof}[1][Proof]{\noindent\textbf{#1.} }{\hfill $\Box$}
\allowdisplaybreaks

\makeatletter  
\begin{document}

\title{\textbf{Front-like entire solutions for monostable reaction-diffusion systems }} 
\author{Shi-Liang Wu$^{a,}$\thanks{ Supported by  the Scientific Research Program Funded by Shaanxi
Provincial Education Department (No. 12JK0860)}
,~Haiyan Wang$^{b}$ \\
$^{a}$Department of Applied Mathematics, Xidian University,\\
 Xi'an, Shaanxi 710071, P.R.China\\
$^{b}$School of Mathematical and Natural Sciences\\
Arizona State University\\
Phoenix, AZ 85069, USA
  }

\date{}

\maketitle

\begin{abstract}

This paper is concerned with front-like entire solutions for  monostable reaction-diffusion systems with cooperative and non-cooperative nonlinearities.
In the cooperative case, the existence and asymptotic behavior of spatially independent solutions (SIS) are first proved.
Combining a SIS and traveling fronts with different wave speeds and directions,
the existence and various qualitative properties of entire solutions are then established using comparison principle. In the non-cooperative case, we introduce two auxiliary cooperative systems and establish some comparison arguments for the three systems.  The existence of entire solutions is then proved
via the traveling fronts and SIS of the auxiliary systems.
Our results are applied to some biological
and epidemiological models. To the best of our knowledge, it is the first work to study the  entire solutions of non-cooperative reaction-diffusion systems.


\textbf{AMS Subjective Classifications (2000): }35K57; 35B40; 34K30; 92D25; 92D30
\end{abstract}

\newpage

\section{Introduction}

\noindent

This paper is concerned with entire solutions of the following $m$-dimensional reaction-diffusion system in $\mathbb{R}^N$:
\begin{equation}\label{eq1.1}
u_t=D\Delta u+f(u),{\ \ }x\in\mathbb{R}^N,t\in\mathbb{R},
\end{equation}
where $m,N\in\mathbb{N}$,
$$u=(u_1,\cdots,u_m),{\ }f=(f_1,\cdots,f_m),{\ }D={\rm diag}(d_1,\cdots,d_m),$$
and $(d_1,\cdots,d_m)\gg{\bf0}:=(0,\cdots,0)\in\mathbb{R}^m$.
Here and in what follows, we always use the usual notations for the standard
ordering in $\mathbb{R}^m$.
As usual, system \eqref{eq1.1} is said to be cooperative on $I\subseteq \mathbb{R}^m$ if each $f_i(u)$
is non-decreasing in $u_j$  on $I$  for $1\leq j\neq i\leq m$; otherwise, it is said to be non-cooperative  on $I$.


One important topic for reaction-diffusion systems is the traveling wave solution that describes the phenomenon of wave propagation.
 In the past decades, many studies have led to almost complete
description of traveling wave solutions of \eqref{eq1.1} with cooperative nonlinearity  \cite{Volpert,Tsai,Li,Weinberger1,Weinberger2}.
For example, Volpert et al. \cite{Volpert}
gave a complete result about the monostable and bistable traveling fronts, and 
Tsai \cite{Tsai} investigated the global exponential stability of the bistable traveling fronts.
In a series of papers, Weinberger, Lewis and Li \cite{Li,Weinberger1,Weinberger2} studied
spreading speeds and traveling fronts for general cooperative recursion
systems.  Related results on scalar non-monotone evolution equations, we refer to \cite{Fangz2,Hsuzhao,Lilw,Ma,Wang1,Wang,wuliu,wuliu1,wuliu2}.

In addition to traveling wave solutions, another important topic in
diffusion systems is the interactions of them,
which is crucially related to the pattern formation problem. We refer to \cite{Ei1,Ei2,Kawahara,Moritan1} for more details.
Mathematically, this phenomenon can be described by the so-called front-like {\it entire solution} that is defined for all space and time and behaves like a combination of traveling fronts as $t\rightarrow-\infty$.
 On the other hand, from the dynamical points of view,  the study
of entire solutions is essential for a full understanding of the transient dynamics and the structures
of the global attractor \cite{Moritan}.
In the recent years, there were many works devoted to
the interactions of traveling fronts and entire solutions for scalar reaction-diffusion (both spatially continuous
and discrete) equations with and without delays, see e.g.,
\cite{Cheng,Chengn,Crooks,Fukao,Guom,Hameln1,Hameln2,Liwangwu,Liliuwang,Moritan,Yagisita,Wangliruan,Wang7,Wangli}.

More recently,
Morita and Tachibana \cite{Moritat}, Guo and Wu \cite{Guowu}, and Wang and Lv \cite{Wangl} and Wu \cite{wu} extended the existence of entire solutions for scalar
equations to some specific {\it two component cooperative} reaction-diffusion model systems. The basic idea in these studies, similar to \cite{Cheng,Guom,Liwangwu,Wangliruan},
is to use traveling fronts propagating from both directions of the $x$-axis to build sub- and supersolutions, and then prove the existence results by employing comparison principle.
Unfortunately, it seems
difficult, if not impossible, to construct such supersolutions
for the {\it $m$-component} reaction-diffusion system \eqref{eq1.1}. In fact, to the best of our knowledge, there has been no results on the entire solutions for general cooperative reaction-diffusion systems and non-cooperative systems.

The purpose of the paper is to consider the entire solutions of system \eqref{eq1.1} with cooperative or non-cooperative nonlinearity.
In the cooperative case, the existence and asymptotic behavior of spatially independent solutions are first proved. Since it is  difficult to use traveling fronts to
construct supersolutions for the general {\it $m$-component} system, we extend the method developed in \cite{Hameln1}
for scalar KPP equations to system \eqref{eq1.1}.
More precisely,  we construct appropriate upper estimates by  virtue of the exact asymptotic behavior of the traveling fronts and  spatially independent solution, and then prove the existence and  qualitative features of entire solutions using comparison principle (Theorems \ref{thm2.9} and \ref{thm2.10}).
Although the method is inspired by the work of Hamel and Nadirashvili \cite{Hameln1}, the technical details are different. In \cite{Hameln1},
the upper estimates were proved by the solution formulation of the linearization of the scalar KPP equation at the trivial equilibrium.
Contrasting to \cite{Hameln1}, we use a general comparison principle to prove
the upper estimates (Lemma \ref{lem2.11}). Recently, the method was successfully applied
in our previous work \cite{wuweng}  to a multi-type SIS nonlocal epidemic model.

For the non-cooperative reaction-diffusion systems, we introduce two
auxiliary cooperative systems, one lies above and another below of system \eqref{eq1.1}, which were used  by Wang \cite{Wang} and several references therein to obtain
the existence of traveling wave solutions, and establish some comparison arguments
for the three systems. Combining the traveling fronts and spatially independent solution of the lower system and their exact asymptotic behavior, we then build
appropriate subsolutions and upper estimates of the auxiliary lower and upper systems using the comparison theorem, respectively,
and prove the existence and qualitative properties of entire solutions of \eqref{eq1.1} with non-cooperative nonlinearity (Theorem \ref{thm3.5}).  To the best of our knowledge, it is the first work to study the  entire solutions of non-cooperative reaction-diffusion systems.


In biology and epidemiology, there are quite a few reaction-diffusion model systems of the form \eqref{eq1.1} with cooperative or non-cooperative nonlinearities.
We shall illustrate our main results by discussing the following models in \cite{Tsai1,Tsai2,cm,Weinberger,Wang}.

{\bf A. A Buffered System.} In \cite{Tsai1,Tsai2},
Tsai and Sneyd presented a buffered system:
\begin{equation}
    \left\{
       \begin{array}{ll}
       \partial_tu_1 =d \Delta u_1
              +g(u_1)+\sum_{i=1}^m[k^-_i(b^0_i-v_i)-k^+_iu_1v_i] ,\\
         \partial_t v_i =d_i\Delta v_i+
                k^-_i(b^0_i-v_i)-k^+_iu_1v_i,\ i=1,\cdots,n,
        \end{array}
    \right.  \label{eq1.3}
\end{equation}
where $d,k^\pm_i,b^0_i>0$ and $d_i\geq0 $ are given  parameters.  They studied the existence, uniqueness and stability of traveling fronts of \eqref{eq1.3} by taking
the typical bistable
nonlinearity for the function $g$, i.e. $g(u_1)=u_1(u_1-a)(1-u_1)$ for some $a\in(0,1).$ Note that \eqref{eq1.3} can be transformed to a cooperative system on
 $\mathbb{R}^+\times\prod_{i=1}^n[0,b^0_i]$  under the change of variable $u_i=b^0_i-v_i$, $i=1,\cdots,n$. Other results related to the buffered system, we refer to \cite{Fangz,Gt,Ka} and the references therein.

{\bf B. An Epidemic Model.}
To study  the fecally-orally transmitted diseases in the European Mediterranean regions,
 Capasso and Maddalena  \cite{cm} introduced the epidemic model:
\begin{equation}
    \left\{
       \begin{array}{ll}
       \partial_tu_1 =d_1 \Delta u_1
               -a_{11}u_1 +a_{12}u_2 ,\\
        \partial_tu_2 =d_2\Delta u_2
                -a_{22}u_2 +g(u_1 ),
        \end{array}
    \right.  \label{eq1.5}
\end{equation}
where 
 $d_1,a_{11},a_{12},a_{22}>0$  and $d_2\geq0$ are given  parameters. The function $g(u_1)$ decribes the infection rate of human under the
assumption that total susceptible human population is constant. In general, $g(\cdot)$ is increasing on $[0,+\infty)$. But, if the ``psychological'' effect is considered (see, e.g., Xiao and
Ruan \cite{xiaor}), then $g(\cdot)$ is a unimodal curve on $[0,+\infty)$, that is, $g(\cdot)$ achieves its
maximum at some $ u_{\max}>0$, and is increasing on $[0,u_{\max}]$ and decreasing on $[u_{\max},+\infty)$.
When $d_2=0$ and $g$ is monotone,  Xu and Zhao \cite{xz} proved the existence, uniqueness and stability of bistable
traveling fronts of (\ref{eq1.5})  and Zhao and Wang \cite{Zhaowang} established the existence and non-existence of monostable traveling fronts. These results were then extended by Wu and Liu \cite{wuliu} to the non-monotone case by constructing two auxiliary monotone integral
equations.

{\bf C. A Population Model.}
Weinberger, Kawasaki and Shigesada \cite{Weinberger} discussed the reaction-diffusion model which describes the
interaction between ungulates with linear density $u_1$ and grass with linear
density  $u_2$:
\begin{equation}\label{eq1.4}
\left\{
\begin{array}{l}
\partial_tu_1 =d_1 \Delta u_1 +  u_1[-\alpha-\delta u_1+r_1u_2],\\
\partial_tu_2 =d_2 \Delta u_2  +r_2u_2[1-u_2 +h(u_1)],
\end{array}\right.
\end{equation}
where $d_1,d_2,r_1,r_2,\alpha,\delta$ are all positive parameters. The function $h(u_1)$ models the increase in the specific growth
rate of the grass due to the presence of ungulates. When the density $u_1$ is small
the net effect of ungulates is increasingly beneficial, but as the density increases
above a certain value, the benefits decrease with increasing. 
In \cite{Weinberger},
Weinberger, Kawasaki and Shigesada  established the spreading speeds for \eqref{eq1.4} by employing comparison
methods. Taking the
non-monotone Ricker function $u_1e^{-u_1}$ as $h(u_1)$, Wang \cite{Wang} further characterized the spreading speed as the
slowest speed of traveling wave solutions.

Throughout this paper, we always make the following assumptions:
 \begin{description}
\item[$\rm(A_0)$] There exists ${\bf K}\gg {\bf 0}$ such that $f({\bf0})=f({\bf K})={\bf0}$, $f\in C^2([{\bf 0},{\bf K}], \mathbb{R}^m)$
 and there is no other positive equilibrium of $f$ between ${\bf 0}$ and ${\bf K}$.
\item[$\rm(A_1)$] One of the following holds:
 \begin{description}
\item[$\rm(a)$] The matrix $f'({\bf0})$ is cooperative and irreducible with $s(f'({\bf0}))>0$, where $$s(f'({\bf0})):=\max\{\Re \lambda: \det(\lambda I-f'({\bf0}))=0\};$$
\item[$\rm(b)$] For each $\lambda\geq0$, $A(\lambda):=D\lambda^2+f'({\bf0})$ is in block lower triangular form,
the first diagonal block has a positive principal eigenvalue 	$M(\lambda)$,
and $M(\lambda)$ is strictly larger than the principal eigenvalues of all other
diagonal blocks. In addition, there is a positive eigenvector $v(\lambda)=(v_1(\lambda),\cdots,v_m(\lambda))\gg0$
of $A(\lambda)$ corresponding to $M(\lambda)$ and $v(\lambda)$ is continuous
with respect to $\lambda$.
\end{description}
\end{description}

We mention that a square matrix is called to be cooperative if all off-diagonal entries are non-negative, and
irreducible if it cannot be placed into block lower-triangular form by simultaneous row/column
permutations (Smith \cite{Smith}).

If $\rm(A_1)(b)$ holds, by the argument of  \cite[Lemma 1.1]{Wang}, there exist two numbers $c_*>0$ and $\lambda_*>0$ such that
\begin{equation}\label{eq1.2}
c_*=\frac{M(\lambda_*)}{\lambda_*}=\inf_{\lambda>0}\frac{M(\lambda)}{\lambda},
\end{equation}
and for any $c>c_*$, there exists  $\lambda_1:=\lambda_1(c)\in(0,\lambda_*)$ such that $M(\lambda_1) =c \lambda_1$ and $M(\lambda)<c \lambda$ for any $\lambda\in(\lambda_1,\lambda_*]$.

If $\rm(A_1)(a)$ holds, then the matrix $A(\lambda)=D\lambda^2+f'({\bf0})$ is also cooperative and irreducible. Hence
$$M(\lambda)=s(A(\lambda)):=\max\{\Re \lambda: \det(\lambda I-A(\lambda))=0\}$$
is a simple eigenvalue of $A(\lambda)$
with an eigenvector $v(\lambda)=(v_1(\lambda),\cdots,v_m(\lambda))\gg0$.
In addition, $M(\lambda)=s(A(\lambda))\geq s(f'({\bf0}))>0$ for any $\lambda\geq0$ (see e.g., \cite[Corollary 4.3.2]{Smith}).
From the argument of \cite[Lemma 2.1]{Fangz}, there also exist $c_*>0$ and $\lambda_*>0$ such that \eqref{eq1.2}  holds,
and for any $c>c_*$, there exists  $\lambda_1:=\lambda_1(c)\in(0,\lambda_*)$ such that $M(\lambda_1) =c \lambda_1$ and $M(\lambda)<c \lambda$ for any $\lambda\in(\lambda_1,\lambda_*]$.

The rest of the paper is organized as follows. In Section 2, we consider the entire solutions of system \eqref{eq1.1} with monostable and cooperative nonlinearity (Theorems \ref{thm2.9} and \ref{thm2.10}).
Section 3 is devoted to the entire solutions of  \eqref{eq1.1} with  monostable and  non-cooperative nonlinearity (Theorem \ref{thm3.5}).
In Section 4,  we apply our abstract results to the above models \eqref{eq1.3}--\eqref{eq1.4}. Finally, conclusions and discussions are given in Section 5.

\section{Entire solutions for cooperative systems}

\noindent

In this section, we consider the entire solutions of \eqref{eq1.1} with monostable and cooperative nonlinearity. In addition to $\rm( A_0)$ and $\rm( A_1)$, we also need the following assumptions:
 \begin{description}
\item[$\rm( A_2)$] System\eqref{eq1.1} is cooperative  on $[{\bf 0},{\bf K}]$, that is,  $\partial_jf_i (u)\geq0$ for all $u\in [{\bf 0},{\bf K}]$ and $1\leq j\neq i\leq m$.
\item[$\rm(A_3)$]  For any $k\in\mathbb{Z}^+$, $\rho_1,\cdots,\rho_k>0$ and $\lambda_1,\cdots, \lambda_k\in[0, \lambda^*]$,
\[  f\big (\min\{  {\bf K},   \rho_1 v(\lambda_1)+\cdots  +\rho_k v(\lambda_k)\}\big ) \leq f'({\bf0})\big [\rho_1 v(\lambda_1)+\cdots + \rho_k v(\lambda_k)\big  ]. \]
\end{description}
Here, $v(\lambda)\gg0$ is the eigenvector
of $A(\lambda)$ corresponding to $M(\lambda)$.

\begin{remark}\label{rem2.1} {\rm It is easily seen that if $f(u)\leq f'({\bf0})u$ for $u\in [{\bf 0},{\bf K}]$, then $\rm(A_3)$ holds spontaneously. We also note that if $f$ is defined on $[ 0,+\infty)^m$, then $\rm(A_3)$ can
be replaced by $\rm(A_3)^*$:
 \begin{description}
\item[$\rm(A_3)^*$] For any $k\in\mathbb{Z}^+$, $\rho_1,\cdots,\rho_k>0$ and $\lambda_1,\cdots, \lambda_k\in[0, \lambda^*]$,
\[  f\big (  \rho_1 v(\lambda_1)+\cdots  +\rho_k v(\lambda_k)\big ) \leq f'({\bf0})\big [\rho_1 v(\lambda_1)+\cdots + \rho_k v(\lambda_k)\big  ]. \]
\end{description}
}
\end{remark}

From the arguments of \cite[Theorem 3.1]{Fangz} and \cite[Theorem 2.1]{Wang}, we have the following result.
\begin{proposition}\label{Pro2.1} Let $\rm( A_0)$--$\rm( A_3)$ hold. For every $c\geq c_*$ and  $\nu \in\mathbb{R}^N$ with $\|\nu\|=1$, \eqref{eq1.1} admits a traveling front $$\Phi_c(\xi)=(\phi_{1,c}(\xi),\cdots,\phi_{m,c}(\xi)),{\ }\xi=x\cdot\nu+ct,$$ which satisfies $\Phi_c(-\infty)={\bf 0}$, $\Phi_c(+\infty)={\bf K}$ and $\Phi_c(\cdot)\gg{\bf 0}$. Furthermore, there holds
 $$\lim_{\xi\rightarrow-\infty}\Phi_{c}(\xi)e^{-\lambda_1(c)\xi}=v(\lambda_1(c))\text{ and }
\Phi_{c}(\xi)\leq v(\lambda_1(c))e^{\lambda_1(c)\xi}\text{ for all }\xi\in\mathbb{R}.$$
\end{proposition}

In the remainder of this section, we first give some comparison theorems for sub and supersolutions of \eqref{eq1.1}. We then state the main results for the cooperative system (Theorems \ref{thm2.9} and \ref{thm2.10}) and establish the existence and asymptotic behavior of spatially independent solutions.
Finally, we prove Theorems \ref{thm2.9} and \ref{thm2.10} by constructing appropriate subsolutions and upper estimates and using a general comparison principle.

\subsection{Preliminaries}

\noindent

Consider the  initial value problem of \eqref{eq1.1} with initial condition:
\begin{eqnarray}
          u(x,\tau)=\varphi(x),{\ \ }x\in\mathbb{R}^N,
      \label{eq2.1}
\end{eqnarray}
where $\tau\in\mathbb{R}$ is an any given constant.

Let $X={\rm BUC}(\mathbb{R}^N,\mathbb{R}^m)$ be the Banach space of all
bounded and uniformly continuous functions from $\mathbb{R}^N$ into
$\mathbb{R}^m$ with the supremum norm $\|\cdot\|_X$.  For simplicity, we denote $W=[{\bf 0},{\bf K}]$ and  $[{\bf 0},{\bf K}]_X =\left\{\phi\in
X:{\bf 0}\leq \phi(x)\leq {\bf K}, x\in \mathbb{R}^N\right\}$. Take $L=\max_{i=1,\cdots,m}\{|\partial_if_i(u)|\big|u\in [{\bf 0},{\bf K}]\}$ and define
\[ Q(u)=(Q_1(u),\cdots,Q_m(u))=f(u)+Lu,\ u\in W.\]
Clearly, $Q(u)$ is non-decreasing in $ u$ for $ u\in W.$
We further define a family of linear operator
\begin{equation}
T(t)={\rm diag}(T_1(t),\cdots,T_m(t)): X\rightarrow X,\ t\geq0,
\label{eq2.2}
\end{equation}
 by
$T_i(0)=I$ and
\[
(T_i(t)\phi)(x)=  e^{-L t}
\int_{\mathbb{R}^N} \Psi_i(y,t)\phi(x-y)dy,\
\forall x\in\mathbb{R}^N,{\ } t>0, {\ }\phi(x)\in {\rm BUC}(\mathbb{R}^N,\mathbb{R}),
\]
where \[
\Psi_i(x,t)=\frac{1}{(4d_i\pi t)^{N/2}}\exp\left\{-\frac{\|x\|^2}{4d_it}\right\},\ i=1,\cdots,m.  \]

The definitions of sub- and supersolutions of (\ref{eq1.1}) are given as follows.
\begin{definition}\label{def2.2}
A continuous function $u=(u_1,\cdots,u_m):\mathbb{R}^N\times  [\tau,+\infty)\rightarrow W$
is called a supersolution of (\ref{eq1.1}) on $[\tau,+\infty)$ if
\begin{eqnarray}
u(x,t)\geq T(t-\tau)u(x,\tau)+ \int_\tau^tT(t-s)Q(u(x,s))ds,{\ \ }\forall x\in\mathbb{R}^N,t>\tau,
 \label{eq2.3}
\end{eqnarray}
A subsolution of (\ref{eq1.1}) is defined by reversing the inequality.
\end{definition}

\begin{remark}\label{re2.3}{\rm
Let $w=(w_1,\cdots,w_m):\mathbb{R}^N\times  [\tau,+\infty)\rightarrow W$ be a continuous function with the property that $w_i$ is
$C^1$ in $t$ and $C^2$ in $x$. It is easy to see that if $w$ satisfies
\begin{eqnarray*}
w_t\geq (or \leq)D\Delta w+f(w),{\ \ } \forall x\in\mathbb{R}^N,t> \tau,
\end{eqnarray*}
then $w$ is a supersolution (or subsolution) of (\ref{eq1.1}) on $[\tau,+\infty)$. }
\end{remark}

By Definition \ref{def2.2}, we have the following results, see e.g., Fang and Zhao \cite{Fangz}.
\begin{lemma}\label{lem2.4} ${\rm (i)}$ For any $\varphi\in [{\bf 0},{\bf K}]_X$,
(\ref{eq1.1}) admits an unique classical solution $u(x,t;\varphi)$ satisfying $u (x,\tau;\varphi)=\varphi(x) $ and ${\bf 0} \leq u(x,t;\varphi)\leq {\bf K}$ for all  $x\in\mathbb{R}^N$ and $t\geq\tau$.\\
${\rm (ii)}$ Let $w^+ (x,t)$ and  $w^- (x,t)$ be a supersolution and a subsolution of (\ref{eq1.1}),
respectively. If $w^+ (\cdot,\tau)\geq
w^- (\cdot,\tau)$, then $w^+ (\cdot,t)\geq
w^- (\cdot,t)$ for all $t\geq\tau$.
\end{lemma}

The following result follows from the standard parabolic estimates (Friedman \cite{Friedman}), see also Wang et al. \cite[Proposition 4.3]{Wangliruan}.
\begin{lemma}\label{lem2.5} Suppose that $u(x,t;\varphi)$ is a solution of (\ref{eq1.1}) with the
initial value $\varphi\in [{\bf 0},{\bf K}]_X$, then there exists a positive constant $M_1$,
independent of $\tau$ and $\varphi$, such that for any
$x\in \mathbb{R}^N$ and $t>\tau+1$,
\[
\left\|\frac{\partial u}{\partial t}(x,t;\varphi)\right\|\leq M_1,{\ }\left\|\frac{\partial^2 u}
{\partial tx_i}(x,t;\varphi)\right\|\leq M_1,{\ }\left\|\frac{\partial^2 u}
{\partial t^2}(x,t;\varphi)\right\|\leq M_1,
 \]
\[
\left\|\frac{\partial u}
{\partial x_i}(x,t;\varphi)\right\|\leq M_1,{\ }\left\|\frac{\partial^2 u}
{\partial x_it}(x,t;\varphi)\right\|\leq M_1,{\ }\left\|\frac{\partial^2 u}
{\partial x_ix_j}(x,t;\varphi)\right\|\leq M_1
\]
\[\left\|\frac{\partial^3 u}
{\partial x_i^2t}(x,t;\varphi)\right\|\leq M_1,{\ }\left\|\frac{\partial^3 u}
{\partial x_i^2x_j}(x,t;\varphi)\right\|\leq M_1,{\ }\forall i,j=1,\cdots,N.
 \]
\end{lemma}

Similar to Lemma \ref{lem2.4}(ii), we have the following result.
\begin{lemma}\label{lem2.6}
Let $u^+\in C\big( \mathbb{R}^N\times [\tau,+\infty) ,[0,+\infty)^m\big)$ and
$$u^-\in
C\big( \mathbb{R}^N\times [\tau,+\infty) , (-\infty,K_1]\times\cdots\times(-\infty,K_m]\big)$$ be such that
$u^+(\cdot,\tau)\geq u^-(\cdot,\tau) $ and
    \begin{eqnarray*}
&&  u_t^+\geq D\Delta u^++f'({\bf0})u^+,{\ \ } \forall x\in\mathbb{R}^N,t> \tau,\\
 && u_t^-\leq D\Delta u^-+f'({\bf0})u^-,{\ \ } \forall x\in\mathbb{R}^N,t> \tau.
\end{eqnarray*}
Then, $u^+(x,t)\geq
u^-(x,t)$ for all $x\in \mathbb{R}^N$ and $t\geq \tau$.
\end{lemma}

\subsection{Main results for cooperative systems}

\noindent

Before to state our main results, we give the following definition and notation.
\begin{definition}\label{def2.8} Let $n\in\mathbb{N}$ and $p,p_0\in \mathbb{R}^n$.
We say that the functions $W_p(x,t)=\big(W_{1;p}(x,t),\cdots,W_{m;p} (x,t)\big)$ converge to  $W_{p_0} (x,t)=\big(W_{1;p_0} (x,t),\cdots,W_{m;p_0} (x,t)\big)$ as $p\rightarrow p_0$ in the sense of the topology $\mathcal{T}$ if, for any compact set $S\subset\mathbb{R}^{N+1}$, the functions $W_p$, $\partial _tW_{p}$, $\partial _{x_i}W_{p}$,
$\partial _{x_i^2}W_{p}$, $i=1,\cdots,N$, converge uniformly in $S$ to $W_{p_0}$, $\partial _tW_{p_0}$, $\partial _{x_i} W_{p_0}$, $\partial _{x_i^2}W_{p_0}$, $i=1,\cdots,N$,
 as $p\rightarrow p_0$.
\end{definition}
{\it Notation}:
 For any $l\in\mathbb{Z}^+$, $\nu_i\in\mathbb{R}^N$, $i=1,\cdots,l$, $A\in{\mathbb R}$ and $a\in{\mathbb R}$, denote the regions $T^i_{A,a}$ and $\tilde{T}^{i}_{A,a}$, $i=1,\cdots,l+1$, by
\begin{eqnarray*}
T^i_{A,a}:=\big\{x\in\mathbb{R}^N\big|x\cdot\nu_i\geq A \big\}\times[a,+\infty),{\ }i=1,\cdots,l,\ T^{l+1}_{A,a}:=\mathbb{R}^N\times[a,+\infty),\\
\tilde{T}^i_{A,a}:=\big\{x\in\mathbb{R}^N\big|x\cdot\nu_i\leq A \big\}\times(-\infty,a],{\ }i=1,\cdots,l,\ \tilde{T}^{l+1}_{A,a}:=\mathbb{R}^N\times(-\infty,a].
\end{eqnarray*}

Now, we state the main results for the cooperative system as follows.
\begin{theorem}\label{thm2.9} Let $\rm( A_0)$--$\rm( A_3)$ hold. Then, for any $l\in\mathbb{Z}^+$, $\nu_1,\cdots,\nu_l\in\mathbb{R}^N$ with $\|\nu_i\|=1$, $h_1,\cdots,h_{l+1}\in\mathbb{R}$, $c_1,\cdots,c_l>c_*$, and $\chi_1,\cdots,\chi_{l+1} \in\{0,1\}$ with $\chi_1+\cdots+\chi_{l+1}\geq2$, there exists an entire solution $U_p(x,t):=\big(U_{1;p} (x,t),\cdots,U_{m;p}  (x,t)\big)$ of \eqref{eq1.1} such that
\begin{eqnarray}
 \underline{u} (x,t)\leq U_p(x,t)\leq \min\big\{ {\bf K},\Pi(x,t)\big\},\quad \forall (x,t)\in\mathbb{R}^{N+1},
\label{eq2.7}
\end{eqnarray}
where
$p:=p_{\chi_1,\cdots,\chi_{l+1} }=\big( \chi_1c_1,\chi_1h_1, \chi_1\nu_1,\cdots,\chi_lc_l,  \chi_lh_l, \chi_l\nu_l, \chi_{l+1} h_{l+1}\big)$ and
\begin{eqnarray*}
 &&\underline{u} (x,t):=\max\Big\{\max_{i=1,\cdots,l}\chi_i\Phi_{c_i} \big(x\cdot\nu_i+c_it+h_i\big),\chi_{l+1} \Gamma(t+h_{l+1})\Big\},\\
 &&\Pi(x,t):= \sum_{i=1}^l \chi_iv(\lambda_1(c_i))
e^{\lambda_1(c_i)(x\cdot\nu_i+c_it+h_i)}+\chi _{l+1} v^* e^{\lambda^*(t+h_{l+1})}.
\end{eqnarray*}
Here, $\Gamma(t)$ is the spatially independent solution of \eqref{eq1.1} decided in Lemma \ref{lem2.7}, $\lambda^*=M(0)$ and $v^*=v(0)$.

Furthermore, the following statements hold:
\begin{description}
\item[$\rm(i)$] ${\bf0}\ll U_p (x,t)\ll {\bf K}$ and $\frac{\partial}{\partial t}U_p (x,t)\gg {\bf0}$ for all $(x,t)\in\mathbb{R}^{N+1}$.

\item[$\rm(ii)$] $\lim_{t\rightarrow-\infty}\sup_{\|x\|\leq A}\|U_p(x,t)\big\|=0$ for any $A\in\mathbb{N}$.

\item[$\rm(iii)$]  If $\chi_{l+1} =0$, then $\lim_{t\rightarrow+\infty}\sup_{\|x\|\leq A}\big\|U_p(x,t)-{\bf K}\big\|=0$ for any $A\in\mathbb{R}_+$, and if $\chi_{l+1} =1$, then $\lim_{t\rightarrow+\infty}\sup_{x\in \mathbb{R}^N}\big\|U_p(x,t)-{\bf K}\big\|=0$.

\item[$\rm(iv)$]If $\chi_{l+1} =1$, then for every $x\in\mathbb{R}^N$,
$$U_p(x,t)\sim \Gamma(t+h_{l+1})\sim v^*e^{\lambda^*(t+h_{l+1})}\text{ as }t\rightarrow-\infty.$$

\item[$\rm(v)$] If $\chi_{l+1} =0$, then for every $x\in\mathbb{R}^N$,
$$U_p(x,t)=O\big(e^{\vartheta(c_1,\cdots,c_l)t}\big) \text{ as }t\rightarrow-\infty,$$
  where $\vartheta(c_1,\cdots,c_l)=\min\big\{c_1\lambda_1(c_1),\cdots,c_l\lambda_1(c_l)\big\}$.

\item[$\rm(vi)$]  For any $(x,t)\in\mathbb{R}^{N+1}$, $U_p(x,t)$ is increasing with respect to $h_i$, $i=1,\cdots,l+1.$

\item[$\rm(vii)$] When   $N=2$ and $l=2$, let us denote $\nu_i=(\cos \theta_i,\sin\theta_i)$, $\theta_i\in[0,2\pi)$, $i=1,2$. If $\chi_1=0$, then
 \begin{align*}
\frac{\partial}{\partial x_1}U_p (x,t)=\left\{
\begin{array}{ll}
\geq0,& \theta_2\in[0,\frac{\pi}{2}]\cup[\frac{3\pi}{2},2\pi]
 ;\\
\leq0 ,&\theta_2\in[\frac{\pi}{2},\frac{3\pi}{2}]
\end{array}
\right.
\end{align*}
and
  \begin{align*}
\frac{\partial}{\partial x_2}U_p (x,t)=\left\{
\begin{array}{ll}
\geq0,& \theta_2\in[0,\pi] ;\\
\leq0 ,&   \theta_2\in[\pi,2\pi].
\end{array}
\right.
\end{align*}
Similar results hold true for $N=l=2$ and $\chi_2=0$.

 \item[$\rm(viii)$]  $U_p(x,t)$ converges to {\bf K} as $h_i\rightarrow+\infty$ in $\mathcal{T}$ and uniformly on $(x,t)\in T^i_{A,a}$ for any $A,a\in\mathbb{R}$, $i=1,\cdots,l+1$.
 \end{description}
\end{theorem}

According to the assumption $\chi_1,\cdots,\chi_{l+1} \in\{0,1\}$ with $\chi_1+\cdots+\chi_{l+1}\geq2$ in Theorem \ref{thm2.9}, we denote the  entire solution $U_{p}(x,t)$ of \eqref{eq1.1} by
\begin{align}\label{p}
U_{p}(x,t):=\left\{
\begin{array}{ll}
U_{p_0}(x,t),&\mbox{if } \big(\chi_1,\cdots,\chi_{l+1} \big)=\big(1,\cdots,1\big);\\
U_{p_i}(x,t),&\mbox{if } \big(\chi_1,\cdots,\chi_{l+1} \big)=\big(1,\cdots,1,0_i,1,\cdots,1\big), {\ }i=1,\cdots,l;\\
U_{p_{l+1}}(x,t),&\mbox{if } \big(\chi_1,\cdots,\chi_{l+1} \big)=(1,\cdots,1,0_{l+1}\big),
\end{array}
\right.
\end{align}
where $p_0=\big( c_1,h_1, \nu_1,\cdots,c_l,  h_l, \nu_l,  h_{l+1}\big)$, $p_{l+1}=\big( c_1,h_1, \nu_1,\cdots,c_l,  h_l, \nu_l,0\big)$ and
 $$p_i=\big(c_1,h_1, \nu_1,\cdots,c_{i-1},h_{i-1}, \nu_{i-1},0,0,0, c_{i+1},h_{i+1},\nu_{i+1},\cdots,c_l,  h_l, \nu_l,h_{l+1}\big),{\ }i=1,\cdots,l.$$
Moreover, we denote
\begin{align}\label{pp}
U_{p}(x,t):=\left\{
\begin{array}{ll}
 U_{p_{i,j}}(x,t),&\mbox{if } \big(\chi_1,\cdots,\chi_{l+1} \big)=\big(1,\cdots,1,0_i,1,\cdots,1,0_j,1,\cdots,1,1\big), \\
  &\quad1\leq i, j \leq l;\\
U_{p_{i,l+1}}(x,t),&\mbox{if }\big(\chi_1,\cdots,\chi_{l+1} \big)=\big(1,\cdots,1,0_i,1,\cdots,1,0_{l+1}\big), {\ }i=1,\cdots,l,
\end{array}
\right.
\end{align}
where
\begin{eqnarray*}
p_{i,j}&= &\big(c_1,h_1, \nu_1,\cdots,c_{i-1},h_{i-1}, \nu_{i-1},0,0,0, c_{i+1},h_{i+1},\nu_{i+1},\cdots,\\
&&\quad c_{j-1},h_{j-1}, \nu_{j-1}, 0,0,0,c_{j+1},h_{j+1},\nu_{j+1},\cdots, c_l,  h_l,\nu_l,h_{l+1}\big), {\ }1\leq i\leq j\leq l,
 \end{eqnarray*}
 and
 $$p_{i,l+1}=\big(c_1,h_1, \nu_1,\cdots,c_{i-1},h_{i-1}, \nu_{i-1},0,0,0, c_{i+1},h_{i+1},\nu_{i+1},\cdots,c_l,  h_l, \nu_l,0 \big),{\ }i=1,\cdots,l.$$

Then we have the following convergence results.
\begin{theorem}\label{thm2.10} Assume $\rm( A_0)$--$\rm( A_3)$. Assume further that
 $f'(u)\leq f'({\bf0})$ for $u\in [{\bf 0},{\bf K}]$. Then, from \eqref{p} and \eqref{pp}, the following properties hold.
\begin{description}
\item[$\rm(i)$]  For any $A\in\mathbb{Z}$ and $a\in\mathbb{R}$, $U_{p_0}(x,t)$ converges to
\begin{align*}
 U_{p_i}(x,t)\ as\ h_i\rightarrow-\infty\ in\ \mathcal{T}, \ and\ uniformly\ on\ (x,t)\in \tilde{T}^i_{A,a},{\ }i=1,\cdots,l+1.
       \end{align*}
   \item[$\rm(ii)$]  For any $A\in\mathbb{R}$ and $a\in\mathbb{R}$, $U_{p_i}(x,t)(i=1,\cdots,l)$ converges to
\begin{align*}
\left\{
\begin{array}{ll}
 U_{p_{i,j}}(x,t)\ as \ h_j\rightarrow-\infty\ in\ \mathcal{T},\ and\ uniformly\ on\ (x,t)\in \tilde{T}^j_{A,a},{\ }1\leq i\neq j\leq l.;\medskip\\
U_{p_{i,l+1}}(x,t)\ as\ h_{l+1}\rightarrow-\infty\ in\ \mathcal{T},\ and  \ uniformly\ on\ (x,t)\in \tilde{T}^{l+1}_{A,a}.
 \end{array}
 \right.
       \end{align*}
 \item[$\rm(iii)$] For any $A\in\mathbb{R}$ and $a\in\mathbb{R}$, $U_{p_{l+1}}(x,t)$ converges to
 \begin{align*}
U_{p_{i,l+1}}(x,t)\ as\ h_i\rightarrow-\infty\ in\ \mathcal{T},\ and \ uniformly\ on \ (x,t)\in \tilde{T}^i_{A,a},{\ }i=1,\cdots,l.
       \end{align*}
\item[$\rm(iv)$] For any $h_1,\cdots,h_l,h_1^*,\cdots,h_l^*\in\mathbb{R}$,
there exists $(x_0,t_0)\in\mathbb{R}^{N+1}$,
depending on $c_1,\cdots,$ $c_l,$  $h_1,\cdots,h_l,h_1^*,\cdots,h_l^*$ such that
$$U_{p_{l+1}}(x,t)=U_{p_{l+1}^* }(x+x_0,t+t_0)\text{ for all }(x,t)\in \mathbb{R}^{N+1}.$$
Here, $p_{l+1}^*=\big( c_1,h_1^*, \nu_1,\cdots,c_l,  h_l^*, \nu_l,0\big).$
\end{description}
\end{theorem}

\subsection{Existence of spatially independent solutions}

\noindent

In this subsection, we consider the spatially independent solutions of \eqref{eq1.1} connecting ${\bf 0}$ and ${\bf K}$, that is, solutions of the following ordinary differential problem:
\begin{eqnarray}
\frac{d \Gamma(t)}{d t}=f\big(\Gamma(t)\big),{\ }t\in\mathbb{R},\label{eq2.4}\\
\Gamma(-\infty)={\bf 0},{\ }\Gamma(+\infty)={\bf K},   \label{eq2.5}
\end{eqnarray}
where $\Gamma=(\Gamma_1,\cdots,\Gamma_m)$ and $f=(f_1,\cdots,f_m)$.
Recall that $W=[{\bf 0},{\bf K}]$.

Note that (\ref{eq2.4}) is a cooperative and irreducible system. The existence of such a heteroclinic orbit $\Gamma(t)$  can be established by using the theory of monotone
dynamical systems (see Smith \cite{Smith} and Zhao \cite{Zhao}). However, these results do not give the exponential decay rate of the solution at minus infinity. To overcome the shortcoming, we shall use the standard technique of monotone iteration scheme to prove the existence and asymptotic behavior of the solutions of (\ref{eq2.4}) and \eqref{eq2.5}.

\begin{lemma}\label{lem2.7} Let $\rm( A_0)$--$\rm( A_3)$ hold. There exists a solution  $\Gamma(t):\mathbb{R}\rightarrow W$ of \eqref{eq2.4} and \eqref{eq2.5}
such that
$$\Gamma'(t)\gg {\bf 0},{\ }\lim_{t\rightarrow-\infty}\Gamma(t)e^{-\lambda^*t}=v^*,\text{ and } \Gamma(t)\leq e^{\lambda^*t}v^*\text{ for all }t\in\mathbb{R},$$
 where $\lambda^*=M(0)$ and $v^*=v(0)$.
\end{lemma}
\begin{proof} Since the method is standard, we only sketch the outline.
 Let $C(\mathbb{R},\mathbb{R}^m)$ be the spaces of continuous vector-valued functions on $\mathbb{R}$. Define the
operator $F=(F_1,\cdots,F_m):C(\mathbb{R},W)\rightarrow C(\mathbb{R},\mathbb{R}^m)$ by
\[
F_i(u)(t)=  \int_{-\infty}^{t}e^{-L(t-s)}Q_i(u)(s) ds,{\ }i=1,\cdots,m.
\]
Recall that
$$L=\max_{i=1,\cdots,m}\{|\partial_if_i(u)|\big|u\in W\}\text{ and }Q_i(u)(t)= f_i(u(t))+L u_i(t),{\ }i=1,\cdots,m.$$ It is easy to verify that
each $Q_i(\cdot)$ is a nondecreasing map form $C(\mathbb{R},W)$ to $C(\mathbb{R},\mathbb{R})$ with respect to the point-wise ordering.
The remainder of the proof is divided into the following there steps.\medskip

\noindent{\it Step  1.}
The following observation is straightforward.
\begin{enumerate}
\item[{\rm(i)}] $F:C(\mathbb{R},W)\rightarrow C(\mathbb{R},W)$;
\item[{\rm(ii)}] $F(\phi)(t)\geq F(\psi)(t)$ for $\phi,\psi\in C(\mathbb{R},W)$ with $\phi(t)\geq \psi(t)$;
\item[{\rm(iii)}]  $F(\phi)(t)$ is increasing in $\mathbb{R}$ for $\phi\in C(\mathbb{R},W)$ with $\phi(t)$ is increasing in $\mathbb{R}$.
\end{enumerate}
\noindent{\it Step 2.}
For any fixed $\varepsilon\in\big(1,2\big)$ and sufficiently large $q>1$, define two functions as follows:
\[
\overline{\phi}(t)=\big(\overline{\phi}_1(t),\cdots, \overline{\phi}_m(t)\big)\text{ and }\underline{\phi}(t)=\big(\underline{\phi}_1(t),\cdots,\underline{\phi}_m(t)\big),
\]
where
$$ \overline{\phi}_i(t)=\min\left\{ K_i,v^*_ie^{\lambda^*t}\right\}\text{ and }\underline{\phi}_i(t)=\max\left\{0, v^*_ie^{\lambda^*t}-qv^*_ie^{\varepsilon\lambda^*t}\right\},{\
}t\in\mathbb{R}.$$
Then, by direct computations, we obtain
\begin{center}
${\bf 0}\leq \underline{\phi}(t)\leq\overline{\phi}(t)\leq {\bf K}$,
$F(\overline{\phi})(t)\leq \overline{\phi}(t)$ and $F(\underline{\phi})(t)\geq\underline{\phi}(t)$ for all $t\in\mathbb{R}$.\end{center}
\noindent {\it Step 3.}
Using the monotone iteration technique, we can show that equation \eqref{eq2.4} admits a solution $\Gamma(t)$ which satisfies
\begin{center}
$\Gamma'(t)\geq{\bf 0}$\  \ and\ \
$
\underline{\phi}(t)\leq \Gamma(t)\leq    \overline{\phi}(t)\text{ for all }   t\in\mathbb{R}.
$
\end{center}
Thus, $\Gamma(-\infty)={\bf 0} $, $ \Gamma(+\infty)\in({\bf 0},{\bf K}]$ and
$$ \lim\limits_{t\rightarrow-\infty}\Gamma(t)e^{-\lambda^*t}=v^*,{\ }{\bf 0}\ll\Gamma(t)\leq e^{\lambda^*t}v^*\text{ for all }t\in\mathbb{R}.$$
Moreover, one can easily verify that $\Gamma(+\infty)={\bf K} $ for all $t\in\mathbb{R}$.

Next, we show that $\Gamma'(t)\gg{\bf 0}$ for all $t\in\mathbb{R}$. Since $\partial_jf_i (u)\geq0$ for all $u\in [{\bf 0},{\bf K}]$ and $1\leq j\neq i\leq m$, by (\ref{eq2.4}), we have
\begin{eqnarray*}
\Gamma_i''(t)&=&\partial_1f_i\big(\Gamma(t)\big)\Gamma_1'(t)+\cdots+\partial_mf_i\big(\Gamma(t)\big)\Gamma_m'(t)\\
  &\geq &\partial_if_i\big(\Gamma(t)\big)\Gamma_i'(t)\\
  &\geq& m_0 \Gamma_i'(t),{\ }\forall t\in\mathbb{R},
\end{eqnarray*}
where $m_0=\min\limits_{i=1,\cdots,m}\big\{\partial_if_i(u)\big|u\in W \big\}$.
Thus, for any $\tau\in\mathbb{R}$, we obtain
\begin{equation}
 \Gamma'_i(t)\geq  \Gamma'_i(\tau)e^{m_0(t-\tau)},{\ }\forall t>\tau,{\ }i=1,\cdots,m.
   \label{eq2.6}
\end{equation}
Suppose for the contrary that
there exist $i_0\in\{1,\cdots,m\}$ and $t_0\in\mathbb{R}$ such that $ \Gamma_{i_0}'(t_0)=0$,
it then follows from (\ref{eq2.6}) that $\Gamma_{i_0}'(\tau)=0$ for all $\tau<t_0$. Thus, $\Gamma_{i_0}(\tau)=\Gamma_{i_0}(t_0)$ for all $\tau\leq t_0$ and hence
$0<\Gamma_{i_0}(t_0)= \Gamma_{i_0}(-\infty)=0 $. This contradiction shows that $\Gamma'(t)\gg {\bf 0}$ for all $t\in\mathbb{R}$. The proof is complete.
\end{proof}\medskip

\subsection{Proofs of Theorems \ref{thm2.9} and \ref{thm2.10} }

\noindent

In this subsection, we will use the results of previous subsections to obtain an appropriate upper estimate for solutions of \eqref{eq1.1} and then prove Theorems \ref{thm2.9} and \ref{thm2.10}.

For any $l,n\in {\mathbb Z}^+$, $\nu_1,\cdots,\nu_l\in\mathbb{R}^N$ with $\|\nu_i\|=1$, $h_1,\cdots,h_{l+1}\in\mathbb{R}$, $c_1,\cdots,c_l>c_*$, and $\chi_1,\cdots,\chi_{l+1} \in\{0,1\}$ with $\chi_1+\cdots+\chi_{l+1}\geq2$, we denote
 \begin{eqnarray*}
&&\varphi^n(x):= \max\left\{\max_{i=1,\cdots,l}\chi_i\Phi_{c_i} \big(x\cdot\nu_i-c_in+h_i\big),\chi_{l+1} \Gamma(-n+h_{l+1})\right\},\\
&&\underline{u} (x,t):=\max\Big\{\max_{i=1,\cdots,l}\chi_i\Phi_{c_i} \big(x\cdot\nu_i+c_it+h_i\big),\chi_{l+1} \Gamma(t+h_{l+1})\Big\},{\ }t\geq-n.
\end{eqnarray*}
Let $U^n(x,t)=\big(U_1^n(x,t),\cdots,U_m^n(x,t)\big)$
be the unique solution of the following initial value problem of \eqref{eq1.1}
\begin{align*}
    \left\{
       \begin{array}{ll}
       u_t=D\Delta u+f(u),{\ }x\in\mathbb{R}^N,t>-n,\\
      u(x,-n)=\varphi^n(x),{\ }x\in\mathbb{R}^N.
 \end{array}
    \right.
\end{align*}
 Then, by Lemma \ref{lem2.4}, we have
$$\underline{u} (x,t) \leq U^n(x,t) \leq {\bf K}
\text{ for all }x\in \mathbb{R}^N\text{ and }t\geq-n.$$
The following result provides the appropriate upper estimate of $U^n(x,t)$.
\begin{lemma}\label{lem2.11} Assume $\rm( A_0)$--$\rm( A_3)$.
The function $U^n(x,t)$ satisfies
\[U^n(x,t)  \leq \min\big\{ {\bf K},\Pi(x,t)\big\}\text{ for all }x\in \mathbb{R}^N \text{ and }t\geq-n,\]
where $\Pi(x,t)$ is defined in Theorem \ref{thm2.10}.
\end{lemma}
\begin{proof} Let $v^+(x,t)=\min\big\{ {\bf K},\Pi(x,t)\big\}.$ From Proposition \ref{Pro2.1} and Lemma \ref{lem2.7}, we have
\begin{eqnarray*}
 v^+(x,-n)&= &\min\big\{ {\bf K},\Pi(x,-n)\big\}\\
 &= &\min\Big\{ {\bf K},\sum_{i=1}^l \chi_iv(\lambda_1(c_i))
e^{\lambda_1(c_i)(x\cdot\nu_i-c_in+h_i)}+\chi _{l+1} v^* e^{\lambda^*(-n+h_{l+1})}\Big\}\\
&\geq &\varphi^n(x)=U^n(x,-n) ,{\ }\forall x\in\mathbb{R}^N.
\end{eqnarray*}
By Lemma \ref{lem2.4}(ii), it is sufficient to show that $v^+(x,t)$ is a supersolution of (\ref{eq1.1}) on $[-n,+\infty)$, that is,
\begin{equation}\label{eq2.010}
v^+(x,t)\geq T(t+n)v^+(x,-n)+ \int_{-n}^tT(t-s)Q(v^+(x,s))ds,\ \forall x\in\mathbb{R}^N,t>-n.
\end{equation}
Note that $Q(u)=f(u)+Lu$ is non-decreasing in $ u$ for $ {\bf 0}\leq u\leq{\bf K}$.
For any $x\in\mathbb{R}^N,t>-n$, we have
\begin{eqnarray*}
&&T_i(t+n)v^+_i(x,-n)+ \int_{-n}^tT_i(t-s)Q_i(v^+(x,s))ds\\
&&\leq T_i(t+n)K_i+ \int_{-n}^tT_i(t-s)Q_i({\bf K})ds\\
&&\leq e^{-L (t+n)}K_i+ \int_{-n}^te^{-L (t-s)} L K_ids=K_i.
\end{eqnarray*}
Consequently,
\begin{equation}\label{eq2.10}
T(t+n)v^+(x,-n)+ \int_{-n}^tT(t-s)Q(v^+(x,s))ds\leq {\bf K},{\  }\forall x\in\mathbb{R}^N,t>-n.
\end{equation}
Note also that $A(0)v^* =\lambda^* v^* $ and
$$A(\lambda_1(c_i)) v(\lambda_1(c_i))=M(\lambda_1(c_i))v(\lambda_1(c_i))=  c_i\lambda_1(c_i)
v(\lambda_1(c_i)),{\ }i=1,\cdots,l.$$
It is easy to see that the function $\Pi(x,t)$ satisfies the linear equation:
\[
\Pi_t= D\Delta \Pi+f'({\bf0})\Pi(x,t).
\]
Then, for any $x\in\mathbb{R}^N,t>-n$,  $\Pi(x,t)$ satisfies the integral equation:
\[
\Pi(x,t)= T(t+n)\Pi(x,-n)+ \int_{-n}^tT(t-s)  \big[f'({\bf0})\Pi(x,s)+L \Pi(x,s)\big]ds.
\]
By the assumption $\rm(A_3)$, we obtain
\begin{eqnarray*}
Q(v^+(x,t))&=&f(v^+(x,t))+Lv^+(x,t)\\
&\leq& f'({\bf0})\Pi(x,t)+Lv^+(x,t)\leq f'({\bf0}) \Pi(x,t) +L\Pi(x,t),
\end{eqnarray*}
and hence
\begin{eqnarray}
&&T(t+n)v^+(x,-n)+ \int_{-n}^tT(t-s)Q(v^+(x,s))ds   \nonumber\\
&&\leq T(t+n) \Pi(x,-n)+ \int_{-n}^tT(t-s)\big[f'({\bf0}) \Pi(x,s)+L \Pi(x,s) \big]ds  \nonumber\\
&&= \Pi(x,t).  \label{eq2.11}
\end{eqnarray}
Combining \eqref{eq2.10} and \eqref{eq2.11},  \eqref{eq2.010} holds and the assertion follows from Lemma \ref{lem2.4}.
 This completes the proof.
\end{proof}
\begin{remark} {\rm
We note that if  $f(u)\leq f'({\bf0})u$ for $u\in [{\bf 0},{\bf K}]$, then Lemma \ref{lem2.11} is a direct consequence of Lemma \ref{lem2.6}. In fact, by $f(u)\leq f'({\bf0})u$ for $u\in [{\bf 0},{\bf K}]$, we have
\[
  U^n_t\leq D\Delta U^n+f'(0)U^n,{\ }\forall x\in\mathbb{R}^N,t>-n.
\]
Noting that $U^n(x,-n)= \varphi^n(x)\leq \Pi(x,-n)$ for all $x\in\mathbb{R}^N$ and
\[
\Pi_t= D\Delta \Pi+f'({\bf0})\Pi(x,t),{\ }\forall x\in\mathbb{R}^N,t>-n.
\]
It follows from Lemma \ref{lem2.6} that $U^n(x,t)  \leq \Pi(x,t)$ and hence $U^n(x,t)  \leq \min\big\{ {\bf K},\Pi(x,t)\big\}$ for all  $x\in \mathbb{R}^N \text{ and }t\geq-n$.}
\end{remark}

Now we give the proofs of Theorem \ref{thm2.9}  and \ref{thm2.10}.\medskip

\noindent {\bf Proof of Theorem \ref{thm2.9}}. By Lemmas \ref{lem2.4} and \ref{lem2.11}, we have
\[
\underline{u} (x,t) \leq
U ^n(x,t)  \leq U^{n+1}(x,t) \leq \min\big\{ {\bf K},\Pi(x,t)\big\}
\]
for all $x\in \mathbb{R}^N$ and $t\geq-n$.
Using the priori estimate of Lemma \ref{lem2.5} and the diagonal extraction process, there exists a subsequence
$\{U ^{n_k}(x,t) \}_{k\in \mathbb{N}}$ of $\{U ^{n}(x,t) \}_{n\in \mathbb{N}}$ such that $U^{n_k}(x,t)$
converges to a function $U_p(x,t)=\big(U _{1;p}(x,t),\cdots,U_{m;p} (x,t)\big)$
in the sense of topology $\mathcal{T}$. Since  $U ^n(x,t)  \leq U^{n+1}(x,t) $ for any $t>-n$,  we have
\begin{center}
$\lim\limits_{n\rightarrow +\infty} U ^n(x,t) =U_p(x,t)$ for any $(x,t)\in\mathbb{R}^{N+1}$.
\end{center}
The limit function is unique, whence all of the functions $U ^n(x,t) $ converge to the function $U_p(x,t)$ in the sense of topology $\mathcal{T}$ as $n\rightarrow+\infty$.
Clearly, $U_p(x,t)$ is an entire solution of \eqref{eq1.1} satisfying \eqref{eq2.7}.

The assertions for parts (ii)-(iii) and (vi)-(viii) are direct consequences of  \eqref{eq2.7}. Therefore, we only prove the results of parts (i), (iv) and (v).

(i) Clearly, $U _p(x,t)\gg0$ for all $(x,t)\in\mathbb{R}^{N+1}$. Since
$$
U ^n(x,t)
\geq \underline{u} (x,t)\geq \underline{u} (x,-n)= \varphi  ^n(x)=U^n(x,-n)
$$
for all $(x,t)\in\mathbb{R}^{N}\times[-n,+\infty)$, by Lemma \ref{lem2.4}, we have $\frac{\partial}{\partial t}U^n(x,t)\geq 0$
for $(x,t)\in\mathbb{R}^{N}\times(-n,+\infty)$. This yields $\frac{\partial}{\partial t}U_p(x,t)\geq 0$
for all $(x,t)\in\mathbb{R}^{N+1}$. Noting that
\begin{eqnarray*}
\frac{\partial^2 U_{i;p}}{\partial t^2}&=&d_i\Delta (U_{i;p})_t+\partial_1f_i\big(U_p\big)(U_{1;p})_t+\cdots+ \partial_mf_i\big(U_p\big)(U_{m;p})_t\\
&\geq &d_i\Delta (U_{i;p})_t+\partial_if_i\big(U_p\big)(U_{i;p})_t\\
 &\geq &d_i\Delta (U_{i;p})_t+m_0(U_{i;p})_t ,{\ }
 i=1,\cdots,m,
\end{eqnarray*}
where
$m_0=\min\limits_{i=1,\cdots,m,u\in W}\partial_if_i(u)$,
we obtain for any $\tau\in\mathbb{R}$,
\begin{equation}\label{eq2.12}
            (U_{i;p})_t(x,t)\geq e^{m_0(t-\tau)}\int_{\mathbb{R}^N} \Psi_i(x-y,t-\tau)(U_{i;p})_t(y,\tau)dy\geq0,{\ }\forall x\in\mathbb{R}^N,t>\tau.
 \end{equation}
Assume, by contradiction, that there exist $i_0\in\{1,\cdots,m\}$ and $(x_0 ,t_0)\in\mathbb{R}^{N+1}$ such that $(U_{i_0;p})_t(x_0 ,t_0)=0$,
it then follows from (\ref{eq2.12}) that $(U_{i_0;p})_t(x_0 ,\tau)=0$ for all $\tau\leq t_0$.
Hence $U_{i_0;p}(x_0 ,t)=U_{i_0;p}(x_0 ,t_0)$ for all $t\leq t_0$, which implies that
$\lim_{t\rightarrow-\infty}U_{i_0;p}(x_0 ,t)=U_{i_0;p}(x_0 ,t_0)$.
But following from (\ref{eq2.7}),
$$U_{i_0;p}(x_0 ,t_0)>0\text{ and }\lim_{t\rightarrow-\infty}U_{i_0;p}(x_0 ,t)=0.$$
This contradiction yields that $\frac{\partial}{\partial t}U_p(x,t)\gg0$ for all $(x,t)\in\mathbb{R}^{N+1}$. \medskip

Next, we show that $U_p(x,t)\ll {\bf K}$ for all  $(x,t)\in\mathbb{R}^{N+1}$. Let $V(x,t)= {\bf K}- U_p(x,t)$, then $ {\bf 0}\leq V(x,t) \leq {\bf K}$ and $V_t(x,t)\ll {\bf 0}$
for all $(x,t)\in\mathbb{R}^{N+1}$ and
\begin{equation}
V_t(x,t)=D \Delta V(x,t)-f({\bf K}- V(x,t)).
\label{eq2.13}
\end{equation}
We claim that $V(x,t)\gg {\bf 0}$ for all  $(x,t)\in\mathbb{R}^{N+1}$. If this is not true, then there exist  $i_0\in\{1,\cdots,m\}$ and $(x_0 ,t_0)\in\mathbb{R}^{N+1}$ such that $ V_{i_0}(x_0 ,t_0)=0$, and hence $\Delta  V_{i_0}(x_0 ,t_0) \geq 0$. It follows from (\ref{eq2.13}) that
\begin{eqnarray*}
0&\leq& d_ {i_0}\Delta  V_{i_0}(x_0 ,t_0) \\
&<& f_{i_0}({\bf K}- V(x_0 ,t_0))\\
& = &f_{i_0}( K_1- V_1(x_0 ,t_0),\cdots, K_{i_0-1}- V_{i_0-1}(x_0 ,t_0),K_{i_0} ,K_{i_0+1}- V_{i_0+1}(x_0 ,t_0),\cdots, K_m )\\
&\leq& f_{i_0}({\bf K})=0,
\end{eqnarray*}
which is a contradiction. Thus $V(x,t)\gg {\bf 0}$ and hence  $U_p(x,t)\ll {\bf K}$ for all  $(x,t)\in\mathbb{R}^{N+1}$.

(iv) When $\chi_{l+1} =1$, by \eqref{eq2.7}, we have
\begin{eqnarray*}
&&\max\left\{\max_{i=1,\cdots,l}\chi_i\Phi_{c_i} \big(x\cdot\nu_i+c_it+h_i\big), \Gamma(t+h_{l+1})\right\} \nonumber\\
&&\leq U_p(x,t)\leq\sum_{i=1}^l \chi_iv(\lambda_1(c_i))
e^{\lambda_1(c_i)(x\cdot\nu_i+c_it+h_i)}+ v^* e^{\lambda^*(t+h_{l+1})}.
\end{eqnarray*}
Noting that
$$\lim_{t\rightarrow-\infty}\Gamma(t)e^{-\lambda^*t}=v^*\text{ and }\lim_{\xi\rightarrow-\infty}\Phi_{c_i}(\xi)e^{-\lambda_1(c_i)\xi}=v(\lambda_1(c_i)),{\ } i=1,\cdots,l,
$$
it suffices to show that $ c\lambda_{1}(c)\geq\lambda^*$ for any $c>c_*$. In fact, since $A(\lambda)\geq A(0)$ for any $\lambda\geq0$, $M(\lambda)\geq M(0)=\lambda^*$ (see, e.g., \cite[Corrollary 4.3.2]{Smith}). In view of $M( \lambda_{1}(c))= c\lambda_{1}(c)$ and $\lambda_{1}(c)>0$ for any $c>c_*$, we obtain $ c\lambda_{1}(c)\geq\lambda^*$ for any $c>c_*$ and the assertion follows. The proof of part (v) is similar to that of part (iv) and omitted. This completes the proof of Theorem \ref{thm2.9}. \medskip

\noindent {\bf Proof of Theorem \ref{thm2.10}}. (i) We only  prove the case that $U_{p_0}(t)$ converges to $U_{p_1}(t)$ in the sense of topology $\mathcal{T}$ as $h_1\rightarrow-\infty$, and uniformly on
$(x,t)\in \widetilde{T}^1_{A,a}.$ The proofs for the other cases are similar.

For $(\chi_1,\cdots,\chi_{l+1} )=(1,\cdots,1)$, we denote $\varphi^n(x)$ by $\varphi^n_{p_0}(x)$  and $U ^n(x,t) $ by $U ^n_{p_0}(x,t) $, respectively. Similarly, when $(\chi_1,\cdots,\chi_{l+1} )=(0,1,\cdots,1)$,
we denote $\varphi^n(x)$ by $\varphi^n_{p_1}(x)$  and $U ^n(x,t) $ by $U ^n_{p_1}(x,t) $, respectively.
Let
$$W ^n(x,t)=U ^n_{p_0}(x,t)-U ^n_{p_1}(x,t), {\ }(x,t)\in\mathbb{R}^{N}\times(-n,+\infty),$$
 then
 ${\bf 0}\leq W ^n(x, t)\leq {\bf K}$ for all $(x,t)\in\mathbb{R}^{N}\times(-n,+\infty)$.
In view of $f'(u)\leq f'({\bf0})$ for all $u\in [{\bf 0},{\bf K}]$, we get
\begin{eqnarray*}
\frac{\partial  W ^n}{\partial t}&  =&D\Delta  W ^n +  f(U ^n_{p_0}(x,t))-f(U ^n_{p_1}(x,t))\\
&  =&D\Delta  W ^n +  f'\big(U ^n_{p_0}(x,t)+(1-\theta_3)W ^n(x,t)\big)W ^n(x,t)\\
&\leq& D\Delta  W ^n+f'({\bf 0})W ^n(x,t),{\ }\forall x\in \mathbb{R}^N,t>-n,
\end{eqnarray*}
where $\theta_3\in(0,1)$.
Define the function
$$\widehat{W}(x,t)=v(\lambda_1(c_1))
e^{\lambda_1(c_1)(x\cdot\nu_1+c_1t+h_1)},{\ }(x,t)\in\mathbb{R}^{N+1}.$$
Since
$$A(\lambda_1(c_1)) v(\lambda_1(c_1))=M(\lambda_1(c_1))v(\lambda_1(c_1))=  c_1\lambda_1(c_1)
v(\lambda_1(c_1)),$$
direct computations show that
\begin{align*}
\frac{\partial  \widehat{W} }{\partial t}= D\Delta  \widehat{W} +f'({\bf 0})\widehat{W}(x,t),{\ }\forall x\in \mathbb{R}^N,t\in\mathbb{R}.
\end{align*}
Moreover, by Proposition \ref{Pro2.1}, we have
\begin{eqnarray*}
W ^n(x,-n)&=&U ^n_{p_0}(x,-n)-U ^n_{p_1}(x,-n)\\
&\leq &\Phi_{c_1} \big(x\cdot\nu_1-c_1n+h_1\big)\\
&\leq& v(\lambda_1(c_1))
e^{\lambda_1(c_1)(x\cdot\nu_1-c_1n+h_1)}=\widehat{W}(x,-n).
\end{eqnarray*}
It then follows from Lemma \ref{lem2.6} that
\begin{center}
$0\leq  W ^n(x,t)=U ^n_{p_0}(x,t)-U ^n_{p_1}(x,t) \leq \widehat{W}(x,t)=v(\lambda_1(c_1))
e^{\lambda_1(c_1)(x\cdot\nu_1+c_1t+h_1)}$
\end{center}
for all $(x,t)\in\mathbb{R}^{N}\times[-n,+\infty)$.
Since $\lim\limits_{n\rightarrow+\infty} U ^n_{p_i}(x,t)= U_{p_i}(x,t)$, $i=0,1$, we get
\[
{\bf0}\leq U_{p_0}(x,t)-U_{p_1}(x,t)\leq v(\lambda_1(c_1))
e^{\lambda_1(c_1)(x\cdot\nu_1+c_1t+h_1)} \text{ for all }(x,t)\in\mathbb{R}^{N+1},
\]
which implies that $U_{p_0}(x,t)$ converges to $U_{p_1}(x,t)$ as $h_1\rightarrow-\infty$ uniformly on
$(x,t)\in \widetilde{T}_{A,a}^1$ for any $A,a\in\mathbb{R}$. For any sequence $h_1^\ell$ with $h_1^\ell\rightarrow-\infty$
as $\ell\rightarrow+\infty$, the functions $U_{p^\ell_0}(x,t)$, $p^\ell_0:=( c_1,h_1^\ell, \nu_1,\cdots,c_l,  h_l, \nu_l,  h_{l+1})$, converge to a solution of \eqref{eq1.1} (up to extraction of
some subsequence) in the sense of topology $\mathcal{T}$, which turns out to be $U_{p_1}(x,t)$. The limit does not depend on the sequence $h_1^{\ell}$,
whence all of the functions $U_{p_0}(x,t)$ converge to $U_{p_1}(x,t)$ in the sense of topology $\mathcal{T}$ as $h_1\rightarrow-\infty$, and the assertion of this part follows.

The proofs of parts (ii)-(iii) are similar to that of part (i), and omitted. Moreover, the proof of part (iv) is straightforward. This completes the proof of Theorem \ref{thm2.10}.

\section{Entire solutions for non-cooperative systems}

\noindent

In this section, we consider the entire solutions of \eqref{eq1.1} with monostable and non-cooperative nonlinearity. We introduce two
auxiliary cooperative reaction-diffusion systems and establish some comparison arguments for the three systems.
Then, we prove
the existence and qualitative properties of entire solutions using the comparison theorem.

Throughout this section, in addition to $\rm(A_0)$ and  $\rm(A_1)$, we also make the following assumptions:
 \begin{description}
 \item[$\rm(A_2)'$] There exist ${\bf K}^\pm=(K_1^\pm,\cdots,K_m^\pm)\gg0$ with $0\ll {\bf K}^- \leq {\bf K} \leq {\bf K}^+$ and two continuous and twice piecewise continuous differentiable functions $f^+,f^-:[{\bf 0},{\bf K}^+]\rightarrow\mathbb{R}^m$ such that $f\in C^2\big([{\bf 0},{\bf K}^+], \mathbb{R}^m\big)$, $f^\pm({\bf0})=f^+({\bf K}^+)=f^-({\bf K}^-)={\bf0}$,  and
     $$ f^-(u)\leq  f(u)\leq f^+(u)\text{ for all }u\in [{\bf 0},{\bf K}^+].$$
 \item[$\rm(A_3)'$]  There is no other positive equilibrium of $f^\pm$ between ${\bf 0}$ and ${\bf K}^\pm$, and $f(u)$ and $f^\pm(u)$ have the same Jacobian matrix $f'({\bf 0})$ at $u={\bf 0}$.
\item[$\rm(A_4)'$]  $\partial_jf_i^\pm (u)\geq0$ for all $u\in [{\bf 0},{\bf K}^+]$, $1\leq j\neq i\leq m$.
 \item[$\rm(A_5)'$]  For any $k\in\mathbb{Z}^+$, $\rho_1,\cdots,\rho_k>0$ and $\lambda_1,\cdots, \lambda_k\in[0, \lambda^*]$,
\[  f^+\big (\min\big \{  {\bf K}^+,   \rho_1 v(\lambda_1)+\cdots  +\rho_k v(\lambda_k)\big \}\big ) \leq f'({\bf0})\big [\rho_1 v(\lambda_1)+\cdots + \rho_k v(\lambda_k)\big  ]. \]
 \end{description}
\begin{remark}\label{rem3.1} {\rm Clearly,  if $f^+(u)\leq f'({\bf0})u$ for $u\in [{\bf 0},{\bf K}^+]$, then $\rm(A_5)'$ holds.
We remark that when \eqref{eq1.1} is cooperative, then $f^\pm=f$ and ${\bf K}^\pm={\bf K}$. We also note that if $f$ is defined on $[ 0,+\infty)^m$, then $\rm(A_5)'$ can
be replaced by $\rm(A_5)^*$:
 \begin{description}
\item[$\rm(A_5)^*$]  For any $k\in\mathbb{Z}^+$, $\rho_1,\cdots,\rho_k>0$ and $\lambda_1,\cdots, \lambda_k\in[0, \lambda^*]$,
\[  f^+\big (  \rho_1 v(\lambda_1)+\cdots  +\rho_k v(\lambda_k)\big ) \leq f'({\bf0})\big [\rho_1 v(\lambda_1)+\cdots + \rho_k v(\lambda_k)\big  ]. \]
\end{description}}
\end{remark}

Denote $W^+= [ {\bf 0}, {\bf K}^+]$. It is easy to verify that for any $\varphi\in [ {\bf 0}, {\bf K}^+]_X$,
system \eqref{eq1.1} admits an unique solution $u(x,t;\varphi)$ satisfying $u (\cdot,\tau;\varphi)=\varphi(\cdot) $ and ${\bf 0} \leq u(x,t;\varphi)\leq {\bf K}^+$ for all  $x\in\mathbb{R}^N$ and $t\geq \tau$.

Now, we consider the following two auxiliary cooperative reaction-diffusion systems
\begin{eqnarray}
u_t=D\Delta u+f^+(u),{\ \ }x\in\mathbb{R}^N,t\in\mathbb{R},\label{eq3.1}\\
u_t=D\Delta u+f^-(u),{\ \ }x\in\mathbb{R}^N,t\in\mathbb{R}.\label{eq3.2}
\end{eqnarray}
Take $\widetilde{ L}=\max_{u\in W^+,i=1,\cdots,m}|\partial_if_i^\pm(u)|$ and define
\[ \widetilde{Q}(u)=(\widetilde{Q}_1(u),\cdots,\widetilde{Q}_m(u))=f(u)+\widetilde{ L}u,\ u\in W^+\]
\[ \widetilde{Q}^\pm(u)=(\widetilde{Q}_1^\pm(u),\cdots,\widetilde{Q}_m^\pm(u))=f^\pm(u)+\widetilde{ L}u,\ u\in W^+.\]
Clearly, $ \widetilde{Q}^\pm(u)$ is non-decreasing in $u$ for $u\in W^+ $ and
$$\widetilde{Q}^-(u)\leq \widetilde{Q}(u) \leq \widetilde{Q}^+(u) \text{ for any }u\in W^+.$$
We further define the operator
$\widetilde{T}(t)=(\widetilde{T}_1(t),\cdots,\widetilde{T}_m(t))$ as (\ref{eq2.2}) by replace $L$ with $\widetilde{ L}$.

The following comparison theorem plays an important role in the proof of our main result for the non-cooperative system.
\begin{lemma}\label{lem3.1}
Let $u,u^\pm\in C(\mathbb{R}^N\times  [\tau,+\infty), W^+)$ be such that
\begin{eqnarray}
u^-(x,t)\leq \widetilde{T}(t-\tau)u^-(x,\tau)+ \int_\tau^t\widetilde{T}(t-s)\widetilde{Q}^-(u^-(x,s))ds ,{\ \ }\forall x\in\mathbb{R}^N,t>\tau,  \label{eq3.3}\\
u(x,t)=\widetilde{T}(t-\tau)u(x,\tau)+ \int_\tau^t\widetilde{T}(t-s)\widetilde{Q}(u(x,s))ds,{\ \ }\forall x\in\mathbb{R}^N,t>\tau,   \label{eq3.4}\\
u^+(x,t)\geq \widetilde{T}(t-\tau)u^+(x,\tau)+ \int_\tau^t\widetilde{T}(t-s)\widetilde{Q}^+(u^+(x,s))ds,{\ \ }\forall x\in\mathbb{R}^N,t>\tau,   \label{eq3.5}
\end{eqnarray}
and $u^-(x,\tau)\leq u(x,\tau)\leq u^+(x,\tau)$. Then, there holds
\[
u^-(x,t)\leq u(x,t)\leq u^+(x,t)\text{ for all }x\in\mathbb{R}^N\text{ and }t>\tau.
\]
\end{lemma}
\begin{proof} We first prove $u(x,t)\leq u^+(x,t)\text{ for all }x\in\mathbb{R}^N$ and $t>\tau$.
 Let $w(x,t)=u(x,t)-u^+(x,t) $ and define
\[
L_i= \max_{u\in W^+,j=1,\cdots,m}\frac{\partial \widetilde{Q}_i^+(u) }{\partial u_j}, {\ }i=1,\cdots,m,\text{ and }[r]_+=\max\{r,0\}\text{ for any }r\in\mathbb{R}.
\]
Since $w(\cdot,\tau)\leq0$ and $\widetilde{Q}^+(u)$ is non-decreasing in $u$ for $u\in W^+$, by \eqref{eq3.4} and \eqref{eq3.5}, we obtain
\begin{eqnarray*}
w_i(x,t)& \leq& \widetilde{T}_i(t-\tau) w_i(x,\tau)+  \int_\tau^t\widetilde{T}_i(t-s)\big[\widetilde{Q}_i(u(x,s))- \widetilde{Q}^+_i(u^+(x,s))\big ]ds\\
& \leq&  \int_\tau^t\widetilde{T}_i(t-s)\big[\widetilde{Q}^+_i(u(x,s))- \widetilde{Q}^+_i(u^+(x,s))\big]ds \\
&=&  \int_\tau^t\widetilde{T}_i(t-s)\left(\int_0^1 \frac{d}{d\theta }\widetilde{Q}_i^+(u^+(x,s) + \theta w(x,s) )d\theta \right)ds\\
&=&  \int_\tau^t\widetilde{T}_i(t-s)\left( \sum _{j=1}^m w_j(x,s)\int_0^1 \frac{\partial }{\partial u_j } \widetilde{Q}_i^+(u^+(x,s) + \theta w(x,s) ) d\theta \right)ds\\
&\leq&  \int_\tau^t\widetilde{T}_i(t-s)\left( L_i \sum _{j=1}^m [w_j(x,s)]_+ \right)ds,{\ }\forall x\in\mathbb{R}^N,t>\tau.
\end{eqnarray*}
Consequently,
\begin{equation}\label{eq3.6}
[w_i(x,t)]_+ \leq\int_\tau^t\widetilde{T}_i(t-s)\left( L_i \sum _{j=1}^m [w_j(x,s)]_+ \right)ds ,{\ }\forall x\in\mathbb{R}^N,t>\tau.
\end{equation}
Let $\varpi(x,t)= \sum _{i=1}^m [w_i(x,t)]_+$. It follows from \eqref{eq3.6} that
\begin{eqnarray*}
\varpi(x,t) &\leq & \sum _{i=1}^m \int_\tau^t\widetilde{T}_i(t-s) L_i \varpi(x,s) ds\\
&\leq&  \int_\tau^t\sum _{i=1}^m \int_{ \mathbb{R}^N }  L_i \Psi_i(x-y,t-s) \varpi(y,s)ds \\
&=&  \int_\tau^t\int_{ \mathbb{R}^N }   P(x-y,t-s) \varpi(y,s)ds,
\end{eqnarray*}
where $ P(y,s)= \sum _{i=1}^m L_i \Psi_i(y,s) $. Using the same argument as in \cite[Lemma 3.2]{Thieme}, we
obtain   $\varpi(x,t) =0$, and hence $u(x,t)\leq u^+(x,t)\text{ for all }x\in\mathbb{R}^N$ and $t>\tau$. Similarly, we can prove that
$u^-(x,t)\leq u(x,t)\text{ for all }x\in\mathbb{R}^N$ and $t>\tau$. This completes the proof.
\end{proof}\medskip

The following result is a direct consequence of Lemma \ref{lem3.1}, see also File \cite{Fife}.
\begin{corollary}\label{cor3.2}
Let $u,u^\pm\in C(\mathbb{R}^N\times  [\tau,+\infty), W^+)$ be such that $u_i,u^\pm_i$ is
$C^1$ in $t$ and $C^2$ in $x$. If
\begin{eqnarray*}
u_t^-\leq D\Delta u^-+f^-(u^-),{\ \ }\forall x\in\mathbb{R}^N,t>\tau,  \\
u_t=D\Delta u+f(u),{\ \ }\forall x\in\mathbb{R}^N,t>\tau,  \\
u_t^+\geq D\Delta u^++f^+(u^+),{\ \ }\forall x\in\mathbb{R}^N,t>\tau,
\end{eqnarray*}
and $u^-(x,\tau)\leq u(x,\tau)\leq u^+(x,\tau)$, then,
\[
u^-(x,t)\leq u(x,t)\leq u^+(x,t)\text{ for all }x\in\mathbb{R}^N,t>\tau.
\]
\end{corollary}


From the argument of Wang \cite[Theorem 2.1]{Wang}, we have the following result.
\begin{proposition}\label{Pro3.3} Let $\rm( A_0)$--$\rm( A_1)$ and $\rm( A_2)'$--$\rm( A_5)'$ hold. For any $c>c_*$ and $\nu \in\mathbb{R}^N$ with $\|\nu\|=1$,
(\ref{eq3.2}) has a non-decreasing traveling wave solution
$$\Phi_c^-(x\cdot\nu+ct)=\big(\phi_{1,c}^-(x\cdot\nu+ct),\cdots,\phi_{m,c}^-(x\cdot\nu+ct)\big),$$
 which satisfies $\Phi_c^-(\cdot)\gg{\bf 0}$, $\Phi_c^-(-\infty)={\bf 0}$, $\Phi_c^-(+\infty)={\bf K}^-$ and
\begin{equation}
\lim_{\xi\rightarrow-\infty}\Phi_{c}^-(\xi)e^{-\lambda_1(c)\xi}=v(\lambda_1(c)),{\ }
\Phi_{c}^-(\xi)\leq v(\lambda_1(c))e^{\lambda_1(c)\xi}\text{ for all }\xi\in\mathbb{R}.
\label{eq3.7}
\end{equation}
Here, $c_*$, $\lambda_1(c)$ and $v(\lambda_1(c))$ are given as in Section 1.
\end{proposition}

We also consider the following ordinary differential system
\begin{eqnarray}
u'(t)=f^-(u),{\ \ }t\in\mathbb{R}.\label{eq3.8}
\end{eqnarray}
By Lemma \ref{lem2.7}, the following result holds.
\begin{lemma}\label{lem3.4} Let  $\rm( A_0)$--$\rm( A_1)$ and $\rm( A_2)'$--$\rm( A_5)'$ hold. There exists a solution $\Gamma^-(t):\mathbb{R}\rightarrow W^+$
of \eqref{eq3.8} which satisfies $\Gamma^-(-\infty)={\bf 0}$ and $\Gamma^-(+\infty)={\bf K}^-$. Furthermore,
$$\frac{d}{dt}\Gamma^-(t)\gg {\bf 0},{\ }\lim_{t\rightarrow-\infty}\Gamma^-(t)e^{-\lambda^*t}=v^*\text{ and } \Gamma^-(t)\leq e^{\lambda^*t}v^*\text{ for all }t\in\mathbb{R},$$
where $\lambda^*=M(0)$ and $v^*=v(0)$.
\end{lemma}

The following theorem contains the main results of this section.
\begin{theorem}\label{thm3.5} Let  $\rm( A_0)$--$\rm( A_1)$ and $\rm( A_2)'$--$\rm( A_5)'$ hold. For any $l\in \mathbb{Z}^+$, $\nu_1,\cdots,\nu_l\in\mathbb{R}^N$ with $\|\nu_i\|=1$, $h_1,\cdots,h_{l+1}\in\mathbb{R}$, $c_1,\cdots,c_l>c_*$, and $\chi_1,\cdots,\chi_{l+1} \in\{0,1\}$ with $\chi_1+\cdots+\chi_{l+1}\geq 1$, there exists an entire solution $U(x,t):=\big(U_{1} (x,t),\cdots,U_{m}  (x,t)\big)$ of \eqref{eq1.1} such that
\begin{eqnarray}
u^-(x,t)\leq U(x,t)\leq \min\big\{ {\bf K}^+,\Pi(x,t)\big\}
\label{eq3.9}
\end{eqnarray}
$\text{ for all }(x,t)\in\mathbb{R}^{N+1},$ where
 \begin{eqnarray*}
&&u^-(x,t)= \max\Big\{\max_{i=1,\cdots,l}\chi_i\Phi_{c_i}^- \big(x\cdot\nu_i+c_it+h_i\big),\chi_ {l+1} \Gamma^-(t+h_{l+1})\Big\},\\
 &&\Pi(x,t)= \sum_{i=1}^l \chi_iv(\lambda_1(c_i))
e^{\lambda_1(c_i)(x\cdot\nu_i+c_it+h_i)}+\chi _{l+1} v^* e^{\lambda^*(t+h_{l+1})} .
\end{eqnarray*}

Furthermore, the following statements hold:
\begin{description}
\item[$\rm(i)$] $U(x,t)\gg0$ for $(x,t)\in\mathbb{R}^{N+1}$ and  $\lim_{t\rightarrow-\infty}\sup_{\|x\|\leq A}\|U(x,t)\big\|=0$ for any $A\in\mathbb{R}_+$.

\item[$\rm(ii)$]  If $\chi_{l+1} =1$, then $\liminf_{ t\rightarrow+\infty}\inf_{ x\in\mathbb{R} }U(x,t)  \geq K^-$ and for every $x\in\mathbb{R}^N$,
$$U(x,t)\sim v^*e^{\lambda^*(t+h_{l+1})}\text{ as }t\rightarrow-\infty.$$

\item[$\rm(iii)$] If $\chi_{l+1} =0$, then  $\liminf_{ t\rightarrow+\infty}\inf_{ \|x\|\leq A }U(x,t)  \geq K^-$ for any $A\in\mathbb{R}_+$ and for every $x\in\mathbb{R}^N$,
$$U(x,t)=O\big(e^{\vartheta(c_1,\cdots,c_l)t}\big) \text{ as }t\rightarrow-\infty,$$
  where $\vartheta(c_1,\cdots,c_l)=\min\big\{c_1\lambda_1(c_1),\cdots,c_l\lambda_1(c_l)\big\}$.
 \end{description}

\end{theorem}
\begin{proof}
Let $W^n(x,t)=\big(W_1^n(x,t),\cdots,W_m^n(x,t)\big)$
be the unique solution of the following initial value problem
\begin{align*}
    \left\{
       \begin{array}{ll}
       u_t=D\Delta u+f(u),{\ }x\in\mathbb{R}^N,t>-n,\\
      u(x,-n)=\widetilde{\varphi}^n(x),{\ }x\in\mathbb{R}^N,
 \end{array}
    \right.
\end{align*}
where
\[\widetilde{\varphi}^n(x):= \max\left\{\max_{i=1,\cdots,l}\chi_i\Phi_{c_i}^- \big(x\cdot\nu_i-c_in+h_i\big),\chi _{l+1}\Gamma^-(-n+h_{l+1})\right\}.
\]
We first show the following claim.\\
{\it Claim.} The function $W^n(x,t)$ satisfies
\begin{eqnarray}
u^-(x,t)\leq W^n(x,t)\leq u^+(x,t):=\min\big\{ {\bf K}^+,\Pi(x,t)\big\}\text{ for all }x\in\mathbb{R}^{N},t>-n.
\label{eq3.10}
\end{eqnarray}

In fact, from Proposition \ref{Pro3.3} and Lemma \ref{lem3.4}, we see that
\[
 u^-(x,-n)=  \widetilde{\varphi}^n(x)=  W^n(x,-n)\leq \min\big\{ {\bf K}^+,\Pi(x,-n)\big\}=u^+(x,-n),\ \forall x\in\mathbb{R}.
\]
By Lemma \ref{lem3.1}, it suffices to show that for any $x\in\mathbb{R}^N,t>-n,$
\begin{equation}\label{eq3.11}
u^-(x,t)\leq \widetilde{T}(t+n)u^-(x,-n)+ \int_{-n}^t\widetilde{T}(t-s)\widetilde{Q}^-(u^-(x,s))ds,
\end{equation}
\begin{equation}\label{eq3.12}
u^+(x,t)\geq \widetilde{T}(t+n)u^+(x,-n)+ \int_{-n}^t\widetilde{T}(t-s)\widetilde{Q}^+(u^+(x,s))ds.
\end{equation}
Now we prove \eqref{eq3.11}.
Note that the function $ \widetilde{u}(x,t):=\chi_j\Phi_{c_j}^- \big(x\cdot\nu_j+c_jt+h_j\big)$ ($j=1,\cdots,l$), satisfies the equation
\begin{align*}
       \widetilde{u}_t=D\Delta \widetilde{u}+f^-(\widetilde{u}),
     \end{align*}
or the integral equation
\[
\widetilde{u}(x,t)=  \widetilde{T}(t+n)\widetilde{u}(x,-n)+ \int_{-n}^t\widetilde{T}(t-s)\widetilde{Q}^-(\widetilde{u}(x,s))ds.
\]
Since $u^-(x,t)\geq \widetilde{u}(x,t) $ for $x\in\mathbb{R}^N,t\geq-n,$ and  $\widetilde{Q}^-(u)=f^-(u)+\widetilde{ L}u$ is non-decreasing in $u$ for $u\in W^+ $,  we have
\begin{eqnarray*}
&&\widetilde{T}(t+n)u^-(x,-n)+ \int_{-n}^t\widetilde{T}(t-s)\widetilde{Q}^-(u^-(x,s))ds\\
&&\geq\widetilde{T}(t+n)\widetilde{u}(x,-n)+ \int_{-n}^t\widetilde{T}(t-s)\widetilde{Q}^-(\widetilde{u}(x,s))ds\\
&&=\widetilde{u}(x,t)   ,         \ \forall  x\in\mathbb{R}^N,t>-n,
\end{eqnarray*}
that is,
\begin{equation}\label{eq3.13}
\widetilde{T}(t+n)u^-(x,-n)+ \int_{-n}^t\widetilde{T}(t-s)\widetilde{Q}^-(u^-(x,s))ds\geq \chi_j\Phi_{c_j}^- \big(x\cdot\nu_i+c_jt+h_j\big).
\end{equation}
Similarly, we can show that for $x\in\mathbb{R}^N,t>-n,$
\begin{equation}\label{eq3.14}
\widetilde{T}(t+n)u^-(x,-n)+ \int_{-n}^t\widetilde{T}(t-s)\widetilde{Q}^-(u^-(x,s))ds \geq \chi_ {l+1} \Gamma^-(t+h_{l+1}).
\end{equation}
Hence, \eqref{eq3.11} follows from \eqref{eq3.13} and \eqref{eq3.14}.

Next, we prove \eqref{eq3.12}.
Since $\widetilde{Q}^+(u)=f^+(u)+\widetilde{ L}u$ is non-decreasing in $u$ for $u\in W^+ $,  we get for $x\in\mathbb{R}^N,t>-n,$
\begin{eqnarray*}
&&\widetilde{T}_i(t+n)u^+_i(x,-n)+ \int_{-n}^t\widetilde{T}_i(t-s)\widetilde{Q}_i^+(u^+(x,s))ds\\
&&\leq e^{-\widetilde{ L} (t+n)}K_i^++ \int_{-n}^te^{-\widetilde{ L} (t-s)}K_i^+ \widetilde{ L} ds=K_i^+,\ i=1,\cdots,m.
\end{eqnarray*}
Consequently,
\begin{equation}\label{eq3.15}
\widetilde{T}(t+n)u^+(x,-n)+ \int_{-n}^t\widetilde{T}(t-s)\widetilde{Q}^+(u^+(x,s))ds\leq {\bf K}^+,{\  }\forall x\in\mathbb{R}^N,t>-n.
\end{equation}
Note that $\Pi(x,t)$ satisfies the integral equation:
\begin{eqnarray}
\Pi(x,t)= \widetilde{T}(t+n)\Pi(x,-n)+ \int_{-n}^t\widetilde{T}(t-s) \big  [f'({\bf0})\Pi(x,s)+\widetilde{ L} \Pi(x,s)\big ]ds.
\label{eq3.16}
\end{eqnarray}
By the assumption $\rm(A_5)'$, we obtain
\[
\widetilde{Q}^+(u^+(x,t))= f^+(u^+(x,t)) + \widetilde{ L}u^+(x,t)\leq f'({\bf0}) \Pi(x,t) + \widetilde{ L}\Pi(x,t).
\]
It follows from \eqref{eq3.16} that
\begin{eqnarray}
&&\widetilde{T}(t+n)u^+(x,-n)+ \int_{-n}^t\widetilde{T}(t-s)\widetilde{Q}^+(u^+(x,s))ds   \nonumber\\
&&\leq \widetilde{T}(t+n) \Pi(x,-n)+ \int_{-n}^t\widetilde{T}(t-s)[f'({\bf0}) \Pi(x,s)+\widetilde{ L} \Pi(x,s) ]ds  \nonumber\\
&&= \Pi(x,t).  \label{eq3.17}
\end{eqnarray}
Combining \eqref{eq3.15} and \eqref{eq3.17}, \eqref{eq3.12} holds. Therefore, the claim follows
 from Lemma \ref{lem3.1}.

Moreover, $W^n(x,t)$ satisfies the regular estimates as in Lemma \ref{lem2.4}, that is,
 there exists a positive constant $M$,
independent of $n$, such that for any
$x\in \mathbb{R}^N$ and $t>-n+1$,
 \[
\left\|\frac{\partial W^n}{\partial t}(x,t )\right\|,{\ }\left\|\frac{\partial^2 W^n}
{\partial tx_i}(x,t )\right\|,{\ }\left\|\frac{\partial^2 W^n}
{\partial t^2}(x,t )\right\|,{\ }\left\|\frac{\partial W^n}
{\partial x_i}(x,t )\right\|,{\ }\left\|\frac{\partial^2 W^n}
{\partial x_it}(x,t )\right\| \leq M,
 \]
 and
\[
\left\|\frac{\partial^2 W^n}
{\partial x_ix_j}(x,t )\right\|,{\ }\left\|\frac{\partial^3 W^n}
{\partial x_i^2t}(x,t )\right\|,{\ }\left\|\frac{\partial^3 W^n}
{\partial x_i^2x_j}(x,t )\right\|\leq M,{\ }\forall i,j=1,\cdots,N.
 \]
By using the diagonal extraction process, there exists a subsequence
$\{W ^{n_k}(x,t) \}_{k\in \mathbb{N}}$ of $\{W ^{n}(x,t) \}_{n\in \mathbb{N}}$ such that $W^{n_k}(x,t)$
converges to a function $$U(x,t)=\big(U _{1}(x,t),\cdots,U_{m} (x,t)\big)$$
in the sense of topology $\mathcal{T}$.
Clearly, $U(x,t)$ is an entire solution of \eqref{eq1.1}. By virtue of \eqref{eq3.10}, we have
\[
u^-(x,t)\leq U(x,t)\leq \min\big\{ {\bf K}^+,\Pi(x,t)\big\}\text{ for all }(x,t)\in\mathbb{R}^{N+1}.
\]

From \eqref{eq3.9}, it is easy to see that the assertion of part (i) holds.
Note that $ c\lambda_{1}(c)\geq\lambda^*$ for any $c>c_*$, and
$$
\lim_{t\rightarrow-\infty}\Gamma^-(t)e^{-\lambda^*t}=v^*,{\ }\lim_{\xi\rightarrow-\infty}\Phi_{c_i}^-(\xi)e^{-\lambda_1(c_i)\xi}=v(\lambda_1(c_i)),{\ }i=1,\cdots,l.$$
The assertions for parts (ii) and (iii) are direct consequences of \eqref{eq3.9}.
The proof is complete.
\end{proof}

\section{Applications }

\noindent

In this section, we apply our main results developed in Sections 2 and 3 to the models \eqref{eq1.3}--\eqref{eq1.4}.

\subsection{A buffered system }

\noindent

 Consider the buffered system \eqref{eq1.3}. For simplicity, we consider the case $n=1$, i.e.
\begin{equation}
    \left\{
       \begin{array}{ll}
       \partial_tu_1 =d_1 \Delta u_1
              +g(u_1)+k_1(b-v_1)-k_2u_1v_1 ,\\
         \partial_t v_1 =d_2\Delta v_1+
                k_1(b-v_1)-k_2u_1v_1,
        \end{array}
    \right.  \label{eq4.1}
\end{equation}
where $d_1,d_2,k_1,k_2,b$ are positive constants. Our choice of the function $g$ is
the typical monostable
nonlinearity, i.e. $g(u_1)=u_1(1-u_1)$. Let $w_1=u_1$ and $w_2=b-v_1$, then \eqref{eq4.1} can be transformed to
\begin{equation}
    \left\{
       \begin{array}{ll}
       \partial_tw_1 =d_1 \Delta w_1
              +w_1(1-w_1)+k_1w_2-k_2w_1(b-w_2) ,\\
         \partial_t w_2 =d_2\Delta w_2-
                k_1w_2+k_2w_1(b-w_2).
        \end{array}
    \right.  \label{eq4.2}
\end{equation}
System \eqref{eq4.2} has only two equilibria ${\bf 0}=(0,0)$ and ${\bf K}=\big(1, k_2b/(k_2+k_1)\big)$ and is cooperative on $ [{\bf 0},{\bf K}]$.
Let $D= {\rm diag}(d_1,d_2)$, and
$$f(w_1,w_2)=\big(w_1(1-w_1)+k_1w_2-k_2w_1(b-w_2),\  -
                k_1w_2+k_2w_1(b-w_2)\big).$$
\begin{theorem}\label{thm4.1} If
$   d_1\geq d_2,\ 1>k_2b\text{ and }k_1\geq k_2, $
then the conclusions of Theorem \ref{thm2.9}  are valid for \eqref{eq4.2}.
\end{theorem}

It is easily seen that
\[f'(0)= \left(
\begin{array}{cc}
1-k_2b &k_1\\
k_2b& -k_1
\end{array}
\right).
\]
Obviously, $f'({\bf0})$ is cooperative and irreducible,
and
$$s(f'({\bf0}))=\frac{1-k_2b- k_1+\sqrt{(1-k_2b- k_1)^2+4k_1}}{2}>0.$$
Hence, the conditions $\rm(A_0)$, $\rm(A_1)(a)$ and  $\rm(A_2)$ hold for \eqref{eq4.2}.
Moreover, for any $\lambda\geq0$,
\[A(\lambda):=D\lambda^2+f'({\bf0})=   \left(
\begin{array}{cc}
d_1\lambda^2+1-k_2b &k_1\\
k_2b& d_2\lambda^2-k_1
\end{array}
\right).
\]
Direct computation shows that
\begin{eqnarray*}
M(\lambda)&=& s(A(\lambda))\\
&=&\frac{d_1\lambda^2+d_2\lambda^2+1-k_2b- k_1+\sqrt{[(d_1-d_2)\lambda^2+1-k_2b+ k_1]^2+4k_2k_1b}}{2}>0,
\end{eqnarray*}
and the
eigenvector $v(\lambda)$ corresponding to $M(\lambda)$ is
\[v(\lambda):=(v_1(\lambda),v_2(\lambda))=\big( M(\lambda)-d_2\lambda^2+k_1 ,k_2b \big)\gg (0,0).
\]
Take  $c_*=\inf_{\lambda>0}\frac{M(\lambda)}{\lambda}$.
Next, we check the condition $\rm(A_3)^*$ (see Remark \ref{rem2.1}).
Note that $   d_1\geq d_2,\ 1>k_2b$ and for any $\lambda\geq0$,
\begin{eqnarray*}
\frac{ v_1(\lambda)}{v_2(\lambda)}&=&  \frac{M(\lambda)-d_2\lambda^2+k_1   }{k_2b}\\
&=&  \frac{ 1}{2k_2b}\left[( d_1-d_2)\lambda^2+1-k_2b+ k_1+\sqrt{[(d_1-d_2)\lambda^2+1-k_2b+ k_1]^2+4k_2k_1b}  \right]\\
&>& \frac{ 1}{2}\left[1-k_2b+ k_1+\sqrt{[1-k_2b+ k_1]^2+4k_2k_1b}  \right]
\geq k_1.
\end{eqnarray*}
For any  $k\in\mathbb{Z}^+$, $\rho_1,\cdots,\rho_k>0$ and $\lambda_1,\cdots, \lambda_k\in[0, \lambda^*]$, denote
\[(z_1,z_2):=\big(\rho_1 v_1(\lambda_1)+\cdots + \rho_k v_1(\lambda_k),\rho_1 v_2(\lambda_1)+\cdots + \rho_k v_2(\lambda_k)\big)\gg(0,0).\]
Consequently, $\rm(A_3)^*$ is equivalent to the following two inequalities
\begin{eqnarray*}
&&z_1(1-z_1)+k_1z_2-k_2z_1(b-z_2) \leq  (1-k_2b)z_1+k_1z_2,  \\
&&  -
                k_1z_2+k_2z_1(b-z_2)\leq k_2b z_1- k_1z_2
\end{eqnarray*}
or
\begin{equation}
z_1 \geq k_2z_2 \text{ and} -k_2z_1z_2\leq 0.
\label{eq4.3}
\end{equation}
Since $ \frac{ v_1(\lambda)}{v_2(\lambda)}\geq k_1$ for any $\lambda\geq0$, we have $z_1/ z_2\geq k_1$. Therefore, \eqref{eq4.3} holds if $k_1\geq k_2 $.

\subsection{An epidemic model }

\noindent

Consider the epidemic model \eqref{eq1.5}. Scaling time and absorbing the appropriate constants into $u_2$,
system  \eqref{eq1.5} can be rewritten as
\begin{equation}
    \left\{
       \begin{array}{ll}
        \partial_t u_1(x,t)=\tilde{d}_1 \Delta u_1(x,t)
               -u_1(x,t)+\gamma u_2(x,t),\\
         \partial_t u_2(x,t)=\tilde{d}_2\Delta u_2(x,t)
                -\beta u_2(x,t)+g(u_1(x,t)),
        \end{array}
    \right.  \label{eq4.4}
\end{equation}
where $\tilde{d}_1=d_1/a_{11}>0$, $\tilde{d}_2=d_2/a_{11}^2>0$, $\gamma=a_{12}/a_{11}^2>0$ and $\beta=a_{22}/a_{11}>0$. For convenience, we denote $\tilde{d}_i$ by $d_i$, $i=1,2.$

 We assume
\begin{description}
\item[$\rm( H_1)$]
$g\in C^2([0,+\infty),[0,+\infty))$, $g(0)=g(k)-\frac{\beta }{\gamma}k=0$, $g(u)>\frac{\beta }{\gamma}u$ for $u\in(0,k)$, and $g(u)\leq g'(0)u$ for $u\in[0,k]$, where  $k>0$ is a constant.
\item[$\rm( H_2)$] One of the following holds:
\begin{description}
\item[$\rm( a)$]  $g(u)$ is increasing for $u>0$;
\item[$\rm( b)$]  There exists a number $u_{\rm max}>0$ such that $g(u)$ is increasing for $0<u\leq u_{\rm max}$ and decreasing for $u>u_{\rm max}$.
\end{description}
\end{description}

Let ${\bf K}= (k, g(k)/\beta)$, $D= {\rm diag}(d_1,d_2)$, and
$$f(u_1,u_2)=\big(-u_1+\gamma u_2,  -\beta u_2+g(u_1)\big).$$
Clearly, $f({\bf0})=f({\bf K})={\bf0}$ and
\[f'(0)= \left(
\begin{array}{cc}
-1&\gamma\\
g'(0)& -\beta
\end{array}
\right).
\]
From $\rm( H_1)$, we see $g'(0)>\frac{\beta}{\gamma}>0$. It is easy to see that $f(u)\leq f'({\bf0})u$ for $u\in [{\bf 0},{\bf K}]$,
$f'({\bf0})$ is cooperative and irreducible,
and
$$s(f'({\bf0}))=\frac{-(\beta+1)+\sqrt{(\beta+1)^2+4(\gamma g'(0)-\beta)}}{2}>0.$$
Thus, the conditions $\rm(A_0)$ and $\rm(A_1)(a)$ hold for \eqref{eq4.4}.
Furthermore, for any $\lambda\geq0$,
\[A(\lambda):=D\lambda^2+f'({\bf0})=  \left(
\begin{array}{cc}
d_1\lambda^2-1&\gamma\\
g'(0)& d_2\lambda^2-\beta
\end{array}
\right)
\]
and
\[
M(\lambda)=s(A(\lambda))=\frac{d_1\lambda^2+d_2\lambda^2-\beta-1+\sqrt{[(d_1\lambda^2-1)-(d_2\lambda^2-\beta)]^2+4\gamma g'(0)}}{2}>0.
\]
 Clearly, $\inf_{\lambda>0}\frac{M(\lambda)}{\lambda}$ exists and denote by $c_*.$

 \begin{theorem}\label{thm4.2} Assume   $\rm( H_1)$. The following statements hold:\\
{\rm (i)} If $\rm( H_2)(a)$ or $\rm( H_2)(b)$ holds and $k\leq u_{\rm max}$, then the conclusions of Theorem \ref{thm2.9}  are valid for \eqref{eq4.4}. If, in addition, $g'(u)\leq g'(0)$ for $u\in[0,k]$, then the conclusions of Theorem \ref{thm2.10}  hold true for (\ref{eq4.4}).\\
{\rm (ii)} If $\rm( H_2)(b)$ holds and $ k>u_{\rm max}$, then the conclusions of Theorem \ref{thm3.5}  hold for (\ref{eq4.4}).
\end{theorem}

If $\rm( H_1)$ and $\rm( H_2)(a)$ or $\rm( H_2)(b)$ hold and $k\leq u_{\rm max}$, then system \eqref{eq4.4} is cooperative on $[{\bf 0},{\bf K}]$. It is easy to verify that $\rm(A_2)$--$\rm(A_3)$  hold. If, in addition, $g'(u)\leq g'(0)$ for $u\in[0,k]$, then $f'(u)\leq f'({\bf0})$ for $u\in [{\bf 0},{\bf K}]$. Therefore, the statement ${\rm (i)}$ of Theorem \ref{thm4.3} holds true.

When $\rm( H_1)$, $\rm( H_2)(b)$ hold and $ k>u_{\rm max}$,  system \eqref{eq4.4} is non-cooperative on $[{\bf 0},{\bf K}]$. Take
 $$u_{\rm min}=\inf\Big\{u\in(0,u_{\rm max}]\Big|g(u)= g\Big(\frac{\gamma}{\beta}g(u_{\rm max})\Big)\Big\}.$$
Clearly, $u_{\rm min}>0$. We define two functions $f^\pm(u)$ as follows:
\[f^\pm(u)=\big(-u_1+\gamma u_2,  -\beta u_2+g^\pm(u_1)\big),\]
where
\[g^+(u_1)= \left\{
\begin{array}{lll}
g(u_1),&u_1\in\big[0,u_{\rm max}\big],\\
g(u_{\rm max}),& u_1\in\big[u_{\rm max},  \frac{\gamma}{\beta}g(u_{\rm max})\big]
\end{array}
\right.
\]
and
\[g^-(u_1)= \left\{
\begin{array}{lll}
g(u_1),&u_1\in\big[0,u_{\rm min}\big],\\
g\big(u_{\rm min}\big),& u_1\in\big[u_{\rm min},  \frac{\gamma}{\beta}g(u_{\rm max})\big].
\end{array}
\right.
\]
Clearly,  $g^+(u_1)\leq g'(0)u_1$ for $u_1\in\big[0,  \frac{\gamma}{\beta}g(u_{\rm max})\big]$. Hence, $f^+(u)\leq f'({\bf0})u$ for $u\in [{\bf 0},{\bf K^+}]$ which yields that $\rm(A_5)'$ holds.
One can further check the conditions $\rm(A_2)'$--$\rm(A_4)'$ with ${\bf K}= \big(k, g(k)/\beta\big)$,
$${\bf K}^+= \Big(\frac{\gamma}{\beta}g(u_{\rm max}),g(u_{\rm max})\Big)\text{ and }
{\bf K}^-= \Big(\frac{\gamma}{\beta}g\big(u_{\rm min}\big),g\big(u_{\rm min}\big)\Big).$$
Therefore,  the statement ${\rm (ii)}$ of Theorem \ref{thm4.3} holds true.

We remark that two specific functions
\[
g_1(u)=\frac{\omega u }{1+\nu u}\text{ and }g_2(u)=\frac{\omega u }{1+\nu u^2},
\]
which have been widely used in the mathematical biology literature, satisfies the above conditions for a wide range
of parameters $\omega$ and $\nu$. In fact, we have the following statements:
\\
(a) if $\omega\gamma>\beta$, then the function
$$f(u_1,u_2)=\big(-u_1+\gamma u_2,  -\beta u_2+g_1(u_1)\big)$$
 satisfies the conditions $\rm(H_1)$ and $\rm( H_2)(a)$ with $k=\frac{\omega\gamma-\beta}{\beta\nu}$; \\
(b) if $\omega\gamma>\beta$, then the function
$$f(u_1,u_2)=\big(-u_1+\gamma u_2,  -\beta u_2+g_2(u_1)\big)$$
 satisfies the conditions $\rm(H_1)$ and $\rm( H_2)(b)$ with $$k=\sqrt{\frac{\omega\gamma-\beta}{\beta\nu}}\text{ and }u_{\rm max}=\sqrt{\frac{1}{\nu}}.$$
Furthermore, it is easy to see that if $ \omega\gamma\leq 2\beta$, then $k\leq u_{\rm max}$, and if $ \omega\gamma>2\beta$, then $k>u_{\rm max}$.

\subsection{A population model }

\noindent

Consider the model \eqref{eq1.4} by taking the
non-monotone Ricker function $u_1e^{-u_1}$ as $h(u_1)$.
Let $w_1=u_1$ and $w_2=u_2-1$, then \eqref{eq1.4} reduces to
\begin{equation}\label{eq4.5}
\left\{
\begin{array}{l}
\partial_tw_1 =d_1 \Delta w_1 +  w_1(r_1-\alpha-\delta w_1+r_1w_2),\\
\partial_tw_2 =d_2 \Delta w_2  +r_2(1+w_2)[-w_2 +h(w_1)],
\end{array}\right.
\end{equation}
where $h(w_1)=w_1e^{-w_1}$ and $d_1,d_2,r_1,r_2,\alpha,\delta$ are all positive parameters. Similar to \cite{Wang}, we assume
\begin{equation}\label{eq4.6}
r_1>\alpha,\ d_1\geq d_2\text{ and }\delta\geq\frac{r_1r_2}{r_1+r_2-\alpha}.
\end{equation}
In the nonnegative quadrant, \eqref{eq4.5} has only two equilibrium ${\bf 0}=(0,0)$ and ${\bf K}=(K_1,K_2)$ which satisfy
\begin{equation}\label{eq4.7}
\begin{array}{l}
r_1K_1e^{-K_1}=\delta K_1+\alpha-r_1\text{ and }
K_2 =K_1e^{-K_1}.
\end{array}
\end{equation}

Let $D= {\rm diag}(d_1,d_2)$ and
$$f(w)=\big(w_1(r_1-\alpha-\delta w_1+r_1w_2), \  r_2(1+w_2)[-w_2 +w_1e^{-w_1}]\big).$$
 For any $\lambda\geq0$,
\[A(\lambda):=D\lambda^2+f'({\bf0})=  \left(
\begin{array}{cc}
d_1\lambda^2+r_1-\alpha&0\\
r_2& d_2\lambda^2-r_2
\end{array}
\right).
\]
Direct computation shows that
$
M(\lambda)=d_1\lambda^2+r_1-\alpha>0
$
and the
eigenvector $v(\lambda)$ corresponding to $M(\lambda)$ is
\[v(\lambda):=(v_1(\lambda),v_2(\lambda))=\big((d_1-d_2)\lambda^2+r_1+r_2-\alpha,\ r_2\big)\gg (0,0).
\]
Hence, the conditions $\rm(A_0)$ and $\rm(A_1)(b)$  hold for \eqref{eq4.5}. Take $c_*=\inf_{\lambda>0}\frac{M(\lambda)}{\lambda}.$
Note that $h(w_1)=w_1e^{-w_1}$ achieves its
maximum at $ h_m = 1$, and is increasing on $[0, h_m]$ and decreasing on $[h_m,+\infty)$.

\begin{theorem}\label{thm4.3} Assume  \eqref{eq4.6}. The following statements hold:\\
{\rm (i)} If $K_1\leq1$, then the conclusions of Theorem \ref{thm2.9}  are valid for (\ref{eq4.5}). \\
{\rm (ii)} If $K_1>1$, then the conclusions of Theorem \ref{thm3.5}  hold true for (\ref{eq4.5}).
\end{theorem}

When $K_1\leq1$, system (\ref{eq4.5}) is a cooperative system on $[{\bf 0},{\bf K}]$, i.e., $\rm(A_2)$ holds.
We need to check the condition $\rm(A_3)^*$ (see Remark \ref{rem2.1}). For any  $k\in\mathbb{Z}^+$, $\rho_1,\cdots,\rho_k>0$ and $\lambda_1,\cdots, \lambda_k\in[0, \lambda^*]$, denote
\[(z_1,z_2):=\big(\rho_1 v_1(\lambda_1)+\cdots + \rho_k v_1(\lambda_k),\rho_1 v_2(\lambda_1)+\cdots + \rho_k v_2(\lambda_k)\big)\gg(0,0).\]
Consequently, $\rm(A_3)^*$ is equivalent to the following two inequalities
\begin{eqnarray}
&&z_1[r_1-\alpha-\delta z_1+r_1z_2]\leq (r_1-\alpha)z_1,  \label{eq4.8} \\
&&  r_2(1+z_2)\big(-z_2 +z_1e^{-z_1}\big)\leq r_2(z_1-z_2)
 \label{eq4.9}
\end{eqnarray}
or
\begin{eqnarray}
&&\delta z_1\geq r_1z_2 ,   \label{eq4.10} \\
&&e^{z_1}(z_1+z_2^2)\geq z_1(1+z_2).  \label{eq4.11}
\end{eqnarray}
Since for any $\lambda\geq0$,
\[
\frac{v_1(\lambda)}{v_2(\lambda)}= \frac{ (d_1-d_2)\lambda^2+r_1+r_2-\alpha  }{r_2} \geq  \frac{ r_1+r_2-\alpha  }{r_2} ,
\]
we have $$\frac{z_1}{z_2}\geq \frac{r_1+r_2-\alpha }{r_2}.$$
Note also that $z_1>0$ and $e^{z_1}>1+z_1$.
Thus, the following two equalities suffice to verify (\ref{eq4.10}) and (\ref{eq4.11}):
\[
\delta\frac{r_1+r_2-\alpha }{r_2} \geq r_1\text{ and }z_1z_2^2+\big(z_1-\frac{1}{2}z_2\big)^2+\frac{3}{4}z_2^2\geq0,
\]
which are true provided that (\ref{eq4.6}) holds. 

If $K_1>1$, system (\ref{eq4.5}) is non-cooperative on $[{\bf 0},{\bf K}]$. Similar to \cite{Wang,Weinberger}, we define two functions $f^\pm(u)$ as follows:
\[f^\pm(w)=\big(w_1(r_1-\alpha-\delta w_1+r_1w_2),  r_2(1+w_2)[-w_2 +h^\pm(w_1)]\big),\]
where
\[h^+(w_1)= \left\{
\begin{array}{lll}
w_1e^{-w_1},&w_1\in[0,1],\\
e^{-1},&w_1>1,
\end{array}
\right.
\]
and
\[h^-(w_1)= \left\{
\begin{array}{lll}
w_1e^{-w_1},&w_1\in[0,h_0],\\
K_1^+e^{-K_1^+},& w_1>h_0.
\end{array}
\right.
\]
Here $K_1^+>K_1$ and $h_0\in(0,1]$ are the unique roots of the equations
\[ \delta K_1^++\alpha-r_1-r_1 h^+(K_1^+)=0\text{ and }h_0e^{-h_0}- K_1^+e^{-K_1^+}=0,  \]
respectively. It is easy to verify that $\rm(A_2)'$--$\rm(A_4)'$ hold with ${\bf K}= \big(K_1,K_1e^{-K_1}\big)$ and
${\bf K}^\pm= \big(K_1^\pm,K_1^\pm e^{-K_1^\pm}\big)$, where $K^-_1\in(0,K_1)$ is the unique root of the equation
\[ \delta K_1^-+\alpha-r_1-r_1 h^-(K_1^-)=0. \]
Next, we check the condition $\rm(A_5)^*$ (see Remark \ref{rem3.1}). Let
\[(z_1,z_2):=\big(\rho_1 v_1(\lambda_1)+\cdots + \rho_k v_1(\lambda_k),\rho_1 v_2(\lambda_1)+\cdots + \rho_k v_2(\lambda_k)\big)\gg(0,0).\]
Consequently, $\rm(A_5)^*$ is equivalent to the following two inequalities
\begin{eqnarray}
&&z_1[r_1-\alpha-\delta z_1+r_1z_2]\leq (r_1-\alpha)z_1, \label{eq4.12} \\
&& r_2(1+z_2)\big(-z_2 +h^+(z_1)\big)\leq r_2(z_1-z_2). \label{eq4.13}
\end{eqnarray}
 Note that \eqref{eq4.8} and \eqref{eq4.9} hold and $h^+(z_1)=z_1e^{-z_1}$ for $z_1\in(0,1]$. To verify the above two inequalities, we only need to show \eqref{eq4.13} holds for $z_1>1$, i.e.,
\[
(1+z_2)(-z_2 +e^{-1})\leq z_1-z_2,
\]
that is,
\[
e(z_1+z_2^2)\geq 1+z_2\text{ for }z_1>1.
\]
It suffices to show that
\[
2(1+z_2^2)\geq 1+z_2,
\]
which holds obviously.
\section{Conclusion and discussion}

\noindent

In this paper, we consider the front-like entire solutions of $m$-dimensional monostable reaction-diffusion systems in $\mathbb{R}^N$.
In the cooperative case, the existence and qualitative properties of entire solutions are established using comparison principle. In the non-cooperative case,
the existence of entire solutions is proved by
citing two auxiliary cooperative systems and establishing some comparison arguments for the three systems.
Uniqueness and stability of entire solutions of such systems seem to
be very interesting and challenging problems.
Besides, the issue of entire solutions of general bistable reaction-diffusion systems remains an open problem.

We mention that the assumption $(d_1,\cdots,d_m)\gg{\bf0}:=(0,\cdots,0)\in\mathbb{R}^m$ (i.e. \eqref{eq1.1} is non-degenerate) is crucial for our main results. When some but not all diffusion coefficients are zero (i.e. \eqref{eq1.1} is partially degenerate), system \eqref{eq1.1} has weak regularity and compactness. For example,
if $d_i=0$ for some $i\in\{1,\cdots,m\}$, then $u_i$ is not smooth enough with respect to $x$ due to zero diffusion coefficient and hence the prior estimate for $u_i$ is not valid (see Lemma \ref{lem2.5}). Recently, in \cite{wu}, we considered the entire solution of the  reaction-diffusion system modeling
man-environment-man epidemics with bistable nonlinearity:
\begin{equation}
    \left\{
       \begin{array}{ll}
        \frac{\partial u(x,t)}{\partial t}= d\frac{\partial^2 u(x,t)}{\partial x^2}
             -u(x,t)+\alpha v(x,t),\\
         \frac{\partial v(x,t)}{\partial t}=
                -\beta v(x,t)+g(u(x,t)).
        \end{array}
    \right.  \label{eq5.2}
\end{equation}
To obtain the entire solution, we established the following prior estimate of solutions of (\ref{eq5.2}), see  \cite[Theorem 3.3]{wu}.
\begin{proposition}\label{Pro5.1}
Suppose that $w(x,t)=(u(x,t),v(x,t))$ is a solution of (\ref{eq5.2}) with
initial value $\varphi\in
[{\bf{0}},{\bf K}]_X$, then there exists a positive constant $M  > 0$ such that for any $\varphi\in
[{\bf{0}},{\bf K}]_X$, $x\in\mathbb{R}$ and $t>1$,
\[
\left|u_t(x, t)\right|\leq M ,{\ }\left|u_{tt}(x, t)\right|\leq M ,{\ }\left|u_{tx}(x, t)\right|\leq M ,{\ }\left|u_x(x, t)\right|\leq M ,
\]
\[\left|u_{xt}(x, t)\right|\leq M ,{\ }\left|u_{xx}(x, t)\right|\leq M ,{\ }\left|u_{xxx}(x, t)\right|\leq M ,\left|u_{xxt}(x, t)\right|\leq M ,
\]
\[\left|v_{t}(x, t)\right|\leq M ,{\ }\left|v_{x}(x, t)\right|\leq M ,{\ }\left|v_{tt}(x, t)\right|\leq M . \]
\end{proposition}
As mention above, $v(x,t)$ in general is not $C^1$ in $x$ when $v(0, \cdot)\in C(\mathbb{R}; [0;K_2])$. Hence, the estimates for $v_x$, $v_{tx}$ and  $u_{xxx}$ are not valid. Here, we correct this mistake. We shall prove that $v$, $v_{t}$ and  $u_{xx}$ possess a property which is similar to a global Lipschitz condition with respect to $x$. In fact, we have the following result.
\begin{proposition}\label{Pro5.2}
Suppose that $w(x,t)=(u(x,t),v(x,t))$ is a solution of (\ref{eq5.2}) with
initial value $\varphi=(\varphi_1,\varphi_2)\in
C\big(\mathbb{R},[{\bf0},{\bf{K}}]\big)$, then there exists a positive constant $M> 0$, independent of $\varphi$, such that for any  $x\in\mathbb{R}$ and $t>1$,
\[
\big|u_t(x, t)\big| ,{\ }\big|u_{tt}(x, t)\big| ,{\ }\big|u_{tx}(x, t)\big|,{\ }\big|u_x(x, t)\big| \leq M ,
\]
\[
\big|u_{xt}(x, t)\big|,{\ }\big|u_{xx}(x, t)\big| ,{\ }\big|u_{xxt}(x, t)\big|\leq M ,
\]
\[  \big|v_t(x, t)\big|,{\ }\big|v_{tt}(x, t)\big| \leq M .
\]
If, in addition, there exists a constant $L'>0$ such that for any $\eta>0$,
$\sup_{x\in \mathbb{R}}|\varphi_2(x+\eta)-\varphi_2(x)|\leq L'\eta$,
then for any $\eta>0$,
$$ \sup_{x\in \mathbb{R},t\geq 1} \big|v(x+\eta,t)-v(x,t)\big| \leq M'\eta,\ \sup_{x\in \mathbb{R},t\geq 1} \left|v_t(x+\eta,t)-
v_t(x,t) \right|  \leq M'\eta,
$$
and
$$ \sup_{x\in \mathbb{R},t\geq 1} \left|u_{xx}(x+\eta,t)-
u_{xx}(x,t) \right|  \leq M'\eta,
$$
where $M'>0$ is a constant which is independent of $\varphi$ and $\eta$.
\end{proposition}
It turns out that the results in \cite{wu} hold  for the bistable partially degenerate system (\ref{eq5.2}). More recently, we have extended the results to a class of two component monostable cooperative partially degenerate reaction-diffusion systems.
 However, it seems difficult to establish such results for general partially degenerate reaction-diffusion systems. Thus, an interesting problem is to adress the entire solutions of general partially degenerate reaction-diffusion systems.

\end{document}